\documentclass[oneside,english]{amsart}
\pdfoutput=1
\usepackage[T1]{fontenc}
\usepackage[latin9]{inputenc}
\usepackage{geometry}
\usepackage{babel}
\usepackage{mathtools}
\usepackage{amsbsy}
\usepackage{amstext}
\usepackage{amsthm}
\usepackage{amssymb}
\usepackage{graphicx}
\usepackage{verbatim}
\usepackage[unicode=true,pdfusetitle, bookmarks=true,bookmarksnumbered=false,bookmarksopen=false, breaklinks=false,pdfborder={0 01},backref=false,colorlinks=true] {hyperref}
\usepackage{cleveref}
\usepackage[all]{xy}
\usepackage{xcolor}
\usepackage{soul}
\makeatletter
\numberwithin{equation}{section}
\numberwithin{figure}{section}
\theoremstyle{plain}
\newtheorem{thm}{\protect\theoremname}[section]
\theoremstyle{plain}

\theoremstyle{plain}
\newtheorem{lem}[thm]{Lemma}

\theoremstyle{definition}
\newtheorem{defn}[thm]{\protect\definitionname}
\theoremstyle{definition}
\newtheorem{exmp}[thm]{Example}
\theoremstyle{definition}

\newtheorem{rem}[thm]{\protect\remarkname}

\newcommand\X{\mathbb X}
\newcommand\Y{\mathbb Y}

\title{Involutive knot Floer homology and bordered modules}
\author{Sungkyung Kang}
\address{Center for Geometry and Physics, Institute for Basic Science}
\email{sungkyung38@icloud.com}
\subjclass[2020]{Primary 57K18; Secondary 57K31, 57R58}
\keywords{Involutive Heegaard Floer homology, Dehn surgery}

\makeatother

\providecommand{\corollaryname}{Corollary}
\providecommand{\definitionname}{Definition}
\providecommand{\remarkname}{Remark}
\providecommand{\theoremname}{Theorem}

\begin{document}

\begin{abstract}
%We prove that a suitable truncation of the involutive knot Floer homology of a knot in $S^3$ and the involutive bordered Heegaard Floer theory of its complement are completely determined up to local equivalence by each other. 
We prove that, up to local equivalences, a suitable truncation of the involutive knot Floer homology of a knot in $S^3$ and the involutive bordered Heegaard Floer theory of its complement determine each other. 
In particular, given two knots $K_1$ and $K_2$, we prove that the $\mathbb{F}_2[U,V]/(UV)$-coefficient involutive knot Floer homology of $K_1 \sharp -K_2$ is $\iota_K$-locally trivial if $\widehat{CFD}(S^3 \backslash K_1)$ and $\widehat{CFD}(S^2 \backslash K_2)$ satisfy a certain condition which can be seen as the bordered counterpart of $\iota_K$-local equivalence. 
We further establish an explicit algebraic formula that computes the hat-flavored truncation of the involutive knot Floer homology of a knot from the involutive bordered Floer homology of its complement. It follows that there exists an algebraic satellite operator defined on the local equivalence group of knot Floer chain complexes, which can be computed explicitly up to a suitable truncation. 
\end{abstract}
\maketitle

\setcounter{tocdepth}{1}
\tableofcontents

\section{Introduction}

Given a closed, connected, and oriented 3-manifold $Y$, the minus-flavored Heegaard Floer theory, defined by Ozsv\'{a}th and Szab\'{o} \cite{ozsvath2004holomorphic}, associates to $Y$ a chain complex $CF^-(Y)$ over the ring $\mathbb{F}_2[U]$, whose homotopy type is an invariant of the oriented diffeomorphism class of $Y$. Furthermore, if we are given a knot $K$ inside $Y$, then the knot Floer theory \cite{ozsvath2008knot,zemke2019link} associates to $K$ a homotopy class of a chain complex $CFK_{UV}(Y,K)$ over the ring $\mathbb{F}_2[U,V]$, from which $CF^-(Y)$ can be recovered by taking the specialization $(U,V)=(1,0)$, or equivalently, $(U,V)=(0,1)$.

Like Seiberg-Witten Floer homology, whose intrinsic $\mathbf{Pin}(2)$-symmetry was used by Manolescu \cite{manolescu2016pin} to disprove the triangulation conjecture in high dimensions, Heegaard Floer theory has an intrinsic $\mathbb{Z}_2$-symmetry, which is induced by the involution
\[
(\Sigma,\boldsymbol\alpha,\boldsymbol\beta,z)\mapsto (-\Sigma,\boldsymbol\beta,\boldsymbol\alpha,z)
\]
on the space of pointed Heegaard diagrams representing the given 3-manifold $Y$. This action, which preserves all relevant counts of holomorphic disks, induces a homotopy-involution $\iota_Y$ on $CF^-(Y)$, which is well-defined up to homotopy, as observed first in \cite{hendricks2017involutive}. Involutive Heegaard Floer theory exploits this involution to give new 3-manifold invariants to define new homology cobordism invariants. Those invariants were then used extensively to solve various problems regarding the structures of homology cobordism groups and knot concordance groups \cite{dai2018infinite,hendricks2018connected,hom2020linear,hendricks2020surgery,alfieri2020connected,hendricks2021quotient}.

Moreover, as observed by Hendricks and Manolescu \cite{hendricks2017involutive}, a similar construction can also be applied to knot Floer theory. Recall that knot Floer homology starts by representing a pair $(Y,K)$ of a 3-manifold $Y$ and an oriented knot $K\subset Y$ as a doubly pointed Heegaard diagram, i.e. Heegaard diagram with two basepoints. Then we have symmetry
\[
(\Sigma,\boldsymbol\alpha,\boldsymbol\beta,z,w)\mapsto (-\Sigma,\boldsymbol\beta,\boldsymbol\alpha,w,z)
\]
on the space of doubly pointed Heegaard diagrams representing $(Y,K)$. However, since the basepoints are swapped to compensate the change of orientation on $K$ occurred by reversing the given orientation on the Heegaard surface $\Sigma$, a half-twist along $K$ is needed to define a well-defined homotopy \emph{skew}-autoequivalence $\iota_K$ of $CFK_{UV}(Y,K)$. Due to the presence of a half-twist in the definition of $\iota_K$, it is no longer a homotopy involution, but satisfies the condition 
\[
\iota_K ^2 \sim \xi_K,
\]
where $\xi_K$ denotes the Sarkar map along $K$. The theory of $CFK_{UV}(Y,K)$ together with $\iota_K$ is called involutive knot Floer homology, which was used to prove the existence of a linearly independent infinite family of rationally slice knots in \cite{hom2020linear}.

On the other hand, given a compact oriented 3-manifold $M$ with a suitably parametrized torus boundary, bordered Heegaard Floer theory \cite{lipshitz2016bordered} associates to $M$ a differential module $\widehat{CFD}(M)$ and an $A_\infty$-module $\widehat{CFA}(M)$ over the torus algebra $\mathcal{A}(T^2)$. When $M$ is the 0-framed exterior $S^3 \backslash K$ of a knot $K\subset S^3$, we know from \cite{kotelskiy2020mnemonic} that the homotopy type of those modules is determined by the homotopy type of the truncation $CFK_{\mathcal{R}}(S^3,K)$ of $CFK_{UV}(S^3,K)$ by taking $UV=0$, and vice versa. Furthermore, we know from \cite{hendricks2019involutivebordered} that mimicking the construction of involutive Heegaard Floer theory defines homotopy equivalences
\[
\begin{split}
    \iota_M &:\widehat{CFDA}(\mathbf{AZ})\boxtimes \widehat{CFD}(M)\rightarrow \widehat{CFD}(M), \\
    \iota_M &:\widehat{CFA}(M)\boxtimes \widehat{CFDA}(\overline{\mathbf{AZ}})\rightarrow \widehat{CFA}(M).
\end{split}
\]
Hence, it is natural to ask how the knot involution $\iota_K$ on $CFK_{\mathcal{R}}(S^3,K)$ is related to the bordered involution $\iota_{S^3 \backslash K}$ of its 0-framed knot complement. The following theorem answers this question in the coarse affirmative, by showing that $\iota_K$ and $\iota_{S^3 \backslash K}$ determine each other up to a certain equivalence relation; this equivalence relation is called the $\iota_K$-local equivalence, which can be seen as the involutive algebraic counterpart of knot concordance.

\begin{thm}
\label{thm:mainthm1}
Given two knots $K_1$ and $K_2$, consider the involutions $\iota_{K_1 \sharp -K_2}$ of $CFK_{\mathcal{R}}(S^3,K_1 \sharp -K_2)$, as well as any choice of bordered involutions $\iota_{S^3 \backslash K_1}\in \mathbf{Inv}_D(S^3 \backslash K_1)$ and $\iota_{S^3 \backslash K_2} \in \mathbf{Inv}_D(S^3 \backslash K_2)$. Then $(CFK_{\mathcal{R}}(S^3,K_1\sharp -K_2),\iota_{K_1\sharp -K_2})$ is $\iota_K$-locally equivalent to the trivial complex if and only if there exists a type-D morphism
\[
g:\widehat{CFD}(S^3 \backslash K_1)\rightarrow \widehat{CFD}(S^3 \backslash K_2)
\]
between type-D modules of 0-framed knot complements, such that the diagram
\[
\xymatrix{
\widehat{CFDA}(\mathbf{AZ})\boxtimes \widehat{CFD}(S^3 \backslash K_1) \ar[d]^{\mathbf{id}\boxtimes g}  \ar[rr]^{\iota_{S^3 \backslash K_1}} & &\widehat{CFD}(S^3\backslash K_1)\ar[d]^{g} \\
\widehat{CFDA}(\mathbf{AZ}) \boxtimes \widehat{CFD}(S^3 \backslash K_2) \ar[rr]^{\iota_{S^3 \backslash K_2}} & & \widehat{CFD}(S^3 \backslash K_2)
}
\]
is homotopy-commutative and the induced chain map 
\[
\begin{split}
    \widehat{CF}(S^3) &\simeq \widehat{CFA}(T_\infty)\boxtimes \widehat{CFD}(S^3 \backslash K_1) \\
    &\xrightarrow{\mathbf{id}\boxtimes g} \widehat{CFA}(T_\infty)\boxtimes \widehat{CFD}(S^3 \backslash K_2) \simeq \widehat{CF}(S^3)
\end{split}
\]
is a homotopy equivalence, and a similar type-D morphism also exists in the opposite direction. Here, $T_\infty$ denotes the $\infty$-framed solid torus, and $S^3 \backslash K_1$ and $S^3 \backslash K_2$ are endowed with the 0-framing on their boundaries. Furthermore, the statement also holds if ``any choice of bordered involutions'' is replaced with ``some choice of bordered involutions''.
\end{thm}

We now consider involutive knot Floer homology for satellite knots. Given two knots $K_1$ and $K_2$ whose knot Floer chain complexes are locally equivalent, it is very unclear whether the satellite knots $P(K_1)$ and $P(K_2)$ should also have locally equivalent knot Floer chain complexes, where $P$ is any pattern in $S^1 \times D^2$. Using \Cref{thm:mainthm1}, we prove the existence of a satellite operator in the local equivalence group of knot Floer chain complexes. 

\begin{thm}
\label{thm:mainthm2}
Let $K_1,K_2$ be knots such that $(CFK_{\mathcal{R}}(S^3,K_1\sharp -K_2),\iota_{K_1\sharp -K_2})$ is $\iota_K$-locally equivalent to the trivial complex. Then for any pattern $P\subset S^1 \times D^2$, $(CFK_{\mathcal{R}}(S^3,P(K_1)\sharp -P(K_2)),\iota_{P(K_1)\sharp -P(K_2)})$ is also $\iota_K$-locally equivalent to the trivial complex.
\end{thm}

A very natural question is then how can one explicitly compute $\iota_K$ from $\iota_{S^3 \backslash K}$. Using the bordered quasi-stabilization constructions, we prove the following theorem which provides a formula to compute the hat-flavored truncation of $\iota_K$ from $\iota_{S^3\backslash K}$ up to orientation reversal.
\begin{thm}
\label{thm:mainthm4}
Let $\nu$ be the longitudinal knot in the $\infty$-framed solid torus $T_\infty$. Then there exists a type-D morphism $$f:\widehat{CFD}(T_\infty,\nu)\rightarrow \widehat{CFDA}(\mathbf{AZ})\boxtimes \widehat{CFD}(T_\infty,\nu)$$ such that for any knot $K$ and for any choice of $\iota_{S^3 \backslash K}\in \mathbf{Inv}_D(S^3 \backslash K)$, the induced map 
\[
\begin{split}
    \widehat{CFK}(S^3,K) &\simeq \widehat{CFA}(S^3 \backslash K)\boxtimes \widehat{CFD}(T_\infty,\nu) \\ 
    &\xrightarrow{\iota_{S^3 \backslash K}^{-1} \boxtimes f} \widehat{CFA}(S^3 \backslash K)\boxtimes \widehat{CFDA}(\overline{\mathbf{AZ}})\boxtimes \widehat{CFDA}(\overline{\mathbf{AZ}})\boxtimes \widehat{CFD}(T_\infty,\nu) \\
    &\xrightarrow{\simeq} \widehat{CFA}(S^3 \backslash K)\boxtimes \widehat{CFD}(T_\infty,\nu) \simeq \widehat{CFK}(S^3,K)
\end{split}
\]
is homotopic to the truncation of either $\iota_K$ or its homotopy inverse $\iota^{-1}_K$ to the hat-flavored complex  $\widehat{CFK}(S^3,K)$ under the natural identification 
\[
\widehat{CFA}(T_\infty,\nu)\boxtimes \widehat{CFD}(S^3\backslash K) \simeq \widehat{CFK}(S^3,K)
\]
induced by the pairing theorem \cite[Theorem 11.19]{lipshitz2018bordered}, where $S^3 \backslash K$ is endowed with the 0-framing on its boundary.
\end{thm}

\Cref{thm:mainthm4} can also be used to explicitly compute $\iota_{S^3 \backslash K}$ for some nontrivial knots $K$. The case when $K$ is the figure-eight knot is computed in \Cref{exmp:figeight}. Note that $\widehat{CFD}(S^3 \backslash K)$ is not rigid, i.e. it has more than one homotopy classes of homotopy autoequivalences; \Cref{exmp:figeight} gives the first example of explicitly computing bordered involutive Floer homology for homotopically non-rigid bordered 3-manifolds.

Furthermore, together with the proof of \Cref{thm:mainthm2}, \Cref{thm:mainthm4} can also be considered as an involutive satellite formula. In particular, given a pattern $P\subset S^1 \times D^2$, if $\widehat{CFDA}((S^1 \times D^2)\backslash P)$ is homotopy-rigid and one already knows the action of $\iota_{S^3 \backslash K}$, then one can explicitly compute the hat-flavored involutive knot Floer homology of the satellite knot $P(K)$. 

\begin{rem}
When $P$ is the $(p,1)$-cabling pattern for some $p>0$, the bimodule $\widehat{CFDA}((S^1 \times D^2)\backslash P)$, with respect to some boundary framings, can be computed from the type DAA trimodule of $S^3 \times (\text{pair-of-pants})$, which was explicitly computed in \cite[Table 1]{hanselman2015calculus}, by taking a box tensor product on its $\rho$-boundary with the type D module of the $\frac{1}{p}$-framed solid torus. It is easy to observe, via manual computation, that the resulting bimodule is homotopy-rigid. Hence \Cref{thm:mainthm4} gives a hat-flavored involutive $(p,1)$-cabling formula, which computes the involutive action of the cable knot $K_{p,1}$ from $\iota_{S^3 \backslash K}$.
\end{rem}

\subsection*{Organization} This article is organized as follows. In \Cref{sec:background}, we recall some results regarding involutive Heegaard Floer homology and bordered Floer homology. In \Cref{sec:freebasepoint}, we develop a theory of involutive knot Floer homology with a free basepoint and discuss its relationship with involutive bordered Floer homology of 0-framed knot complements. In \Cref{sec:fromsurgerytohfk}, we prove \Cref{thm:mainthm1} and use it to prove \Cref{thm:mainthm2}. Finally, in \Cref{sec:formula}, we prove \Cref{thm:mainthm4} and discuss its explicit applications.

\subsection*{Acknowledgements} The author would like to thank Kristen Hendricks, Robert Lipshitz, and JungHwan Park for helpful conversations, and Abhishek Mallick, Monica Jinwoo Kang, and Ian Zemke for numerous helpful comments. This work was supported by Institute for Basic Science (IBS-R003-D1).

\section{Involutive Heegaard Floer homology for knots and 3-manifolds}
\label{sec:background}
We assume that the reader is familiar with Heegaard Floer theory \cite{ozsvath2003absolutely,ozsvath2004holomorphic,ozsvath2006holomorphic,ozsvath2004holomorphicproperties} of knots and 3-manifolds, as well as bordered Heegaard Floer theory \cite{lipshitz2018bordered}. Throughout the paper, we will only work with $\mathbb{F}_2$ coefficients. Furthermore, we will often consider 3-manifolds $M$ endowed with torsion $\mathbf{Spin}^c$ structures. In such cases, the Heegaard Floer chain complexes $CF^- (M,\mathfrak{s})$ and $\widehat{CF}(M,\mathfrak{s})$ are chain complexes of free modules over $\mathbb{F}_2[U]$ and $\mathbb{F}_2$, respectively, and absolutely $\mathbb{Q}$-graded. 

\subsection{Involutive Heegaard Floer homology and $\iota$-complexes}
Recall that the definition of Heegaard Floer homology of any flavor starts with choosing an admissible pointed Heegaard diagram $H=(\Sigma,\boldsymbol\alpha,\boldsymbol\beta,z)$ representing $M$. The theory of involutive Heegaard Floer homology, as defined first in \cite{hendricks2017involutive}, starts by considering the conjugate diagram $\bar{H}=(-\Sigma,\boldsymbol\beta,\boldsymbol\alpha,z)$. Then we have a canonical identification map 
\[
\mathbf{conj}:CF^-(H)\xrightarrow{\simeq} CF^- (\bar{H}).
\]
Since $\bar{H}$ also represents $M$, it is related to $H$ by a sequence of Heegaard moves. Such a sequence induces a homotopy equivalence 
\[
\Phi_{\bar{H},H}:CF^-(\bar{H})\rightarrow CF^-(H).
\]
By the naturality of Heegaard Floer theory \cite{juhasz2012naturality}, the homotopy class of $\Phi_{\bar{H},H}$ does not depend on our choice of a sequence of Heegard moves from $\bar{H}$ to $H$. Thus the homotopy autoequivalence
\[
\iota_M=\Phi_{\bar{H},H}\circ \mathbf{conj}:CF^-(M)\rightarrow CF^-(M)
\]
is well-defined up to homotopy, and the image of its restriction $\iota_{\mathfrak{s}}$ to $CF^-(M,\mathfrak{s})$ is $CF^-(M,\bar{\mathfrak{s}})$. In particular, when $\mathfrak{s}$ is self-conjugate, i.e. spin, then $\iota_{M,\mathfrak{s}}$ is a homotopy autoequivalence of $CF^-(M,\mathfrak{s})$.

The involution $\iota_M$ satisfies the following properties.
\begin{itemize}
    \item $\iota_M ^2 \sim \mathbf{id}$.
    \item The localized map $U^{-1}\iota_M$ is homotopic to identity.
\end{itemize}
Inspired by the above properties, the notion of $\iota$-complex was defined in \cite{hendricks2018connected} as follows. An $\iota$-complex is a pair $(C,\iota)$ which satisfies the following properties.
\begin{itemize}
    \item $C$ is a chain complex of finitely generated free modules over $\mathbb{F}_2[U]$, such that the localized complex $U^{-1}C$ has homology $\mathbb{F}_2[U^{\pm 1}]$.
    \item $\iota$ is a homotopy autoequivalence of $C$ such that $\iota^2 \sim \mathbf{id}$.
\end{itemize}
Furthermore, given two chain complexes $C_1$ and $C_2$ of modules over $\mathbb{F}_2[U]$, a chain map $f:C_1\rightarrow C_2$ is said to be a \emph{local map} if the localized map $U^{-1}f:U^{-1}C_1\rightarrow U^{-1}C_2$ induces an injective map in homology. Given two $\iota$-complexes $(C_1,\iota_1)$ and $(C_2,\iota_2)$, a local map $f:C_1\rightarrow C_2$ is said to be a \emph{$\iota$-local map} if $\iota_2 \circ f \sim f \circ \iota_1$. If $\iota$-local maps between $(C_1,\iota_1)$ and $(C_2,\iota_2)$ exist in both directions, we say that the given two $\iota$-complexes are \emph{$\iota$-locally equivalent}. The set of $\iota$-local equivalence classes of $\iota$-complexes forms a group $\mathfrak{I}$ under the tensor product operation, which is called the \emph{local equivalence group}.

The notion of $\iota$-complexes and local equivalences between them can be weakened, as shown in \cite{dai2018infinite}, in the following way. An almost $\iota$-complex is a pair $(C,\iota)$ which satisfies the following properties.
\begin{itemize}
    \item $C$ is a chain complex of finitely generated free modules over $\mathbb{F}_2[U]$, such that the localized complex $U^{-1}C$ has homology $\mathbb{F}_2[U^{\pm 1}]$.
    \item $\iota:C\rightarrow C$ is a chain map of chain complexes of $\mathbb{F}_2$-vector spaces, such that $\mathbf{Im}(\partial \iota + \iota \partial) \subset \mathbf{Im}(U)$ and $\iota^2 \sim \mathbf{id} \mod U$.
\end{itemize}
Given almost local $\iota$-complexes $(C_1,\iota_1)$ and $(C_2,\iota_2)$, a local map $f:C_1\rightarrow C_2$ is an \emph{almost local map} if $\iota_2 \circ f + f \circ \iota_1 \sim 0 \mod U$. If almost local maps exist in both directions, we say that the given two almost $\iota$-complexes are \emph{almost locally equivalent}. Again, the set of almost local equivalences of almost $\iota$-complexes form a group $\hat{\mathfrak{I}}$, which is called the \emph{almost local equivalence group}. The construction of involutive Heegaard Floer homology gives a canonical map
\[
\Phi^3 _{\mathbb{Z}_2} \xrightarrow{f} \mathfrak{I} \xrightarrow{g} \hat{\mathfrak{I}},
\]
where $\Phi^3 _{\mathbb{Z}_2}$ denotes the homology cobordism group of $\mathbb{Z}_2$-homology spheres, $f$ maps a homology cobordism class $[Y]$ of a $\mathbb{Z}_2$-homology sphere $Y$ to its involutive Heegaard Floer homology $[CF^- (Y,[0]),\iota_{Y,[0]}]$ for the unique spin structure $[0]$ on $Y$, and $g$ is the canonical forgetful map.

\begin{rem}
\label{rem:generalizedlocal}
The definition of $\iota$-local maps, local equivalences, and their ``almost'' versions also work when we drop the condition that $U^{-1}C$ is homotopy equivalent to $\mathbb{F}_2[U^{\pm 1}]$. We will sometimes use this generalized notion throughout this paper.
\end{rem}

\subsection{Involutive knot Floer homology and $\iota_K$-complexes}
The involutive theory for knot Floer homology is a bit more complicated than the 3-manifold case. For simplicity, we only consider knots $K$ in $S^3$. Consider a doubly pointed Heegaard diagram $H=(\Sigma,\boldsymbol\alpha,\boldsymbol\beta,z,w)$ representing $(S^3,K)$. By counting holomorphic disks while recording their algebraic intersection numbers with $z$ and $w$ by formal variables $U$ and $V$, respectively, one gets an absolutely $\mathbb{Z}$-bigraded (called \emph{Alexander} and \emph{Maslov} grading, respectively) chain complex $CFK_{UV}(S^3,K)$ of finitely generated free modules over the ring $\mathbb{F}_2[U,V]$.

Consider the conjugate diagram $\bar{H}=(-\Sigma,\boldsymbol\beta,\boldsymbol\alpha,w,z)$ of $H$; note that, in addition to flipping the orientation of $\Sigma$ and exchanging $\alpha$ and $\beta$ curves, we are also exchanging the basepoints $z$ and $w$. Then, as in the 3-manifold case, we have a canonical conjugation map 
\[
\mathbf{conj}:CFK_{UV}(H)\rightarrow CFK_{UV}(\bar{H}),
\]
which is a chain skew-isomorphism, i.e. intertwines the actions of $U$ and $V$ on its domain with the actions of $V$ and $U$ on its codomain. Then we consider a self-diffeomorphism of $S^3$ that acts on a tubular neighborhood of $K$ by a ``half-twist'', so that it fixes $K$ setwise and maps $z$ and $w$ to $w$ and $z$, respectively. It induces a chain isomorphism 
\[
\phi_{\ast}:CFK_{UV}(\bar{H})\rightarrow CFK_{UV}(\phi(\bar{H})).
\]
Now, the diagrams $\phi(\bar{H})$ and $H$ both represent the knot $K$ together with two prescribed basepoints $z$ and $w$ on $K$, so they are related by a sequence of Heegaard moves. Such a sequence induces a homotopy equivalence 
\[
\Phi_{\phi(\bar{H}),H}:CFK_{UV}(\phi(\bar{H}))\rightarrow CFK_{UV}(H),
\]
whose homotopy class is independent of our choice of a sequence of Heegaard moves from $\phi(\bar{H})$ to $H$, due to naturality. Thus we have a homotopy skew-equivalence 
\[
\iota_K =\Phi_{\phi(\bar{H}),H}\circ \phi_{\ast}\circ \mathbf{conj}:CFK_{UV}(S^3,K)\rightarrow CFK_{UV}(S^3,K),
\]
which is well-defined up to homotopy. Note that such a construction can also be applied for links as well; given a link $L$, where each component $K\subset L$ has one $z$-basepoint and one $w$-basepoint (which correspond to formal variables $U_K$ and $V_K$), following the above construction gives a homotopy skew-equivalence $\iota_L$ which intertwines the actions of $U_K$ and $V_K$ for each component $K$.

The homotopy skew-equivalence $\iota_K$ satisfies the following properties, as shown in \cite{zemke2019connected}.
\begin{itemize}
    \item $\iota_K ^2 \sim 1+\Phi\Psi \sim 1+\Psi\Phi$, where $\Phi$ and $\Psi$ are the formal derivatives of the differential $\partial$ of $CFK_{UV}(S^3,K)$ with respect to the formal variables $U$ and $V$, respectively.
    \item The localized map $(U,V)^{-1}$ is homotopic to identity.
\end{itemize}
Using the above properties, the notion of $\iota_K$-complexes was defined in \cite{zemke2019connected} as follows. An $\iota_K$-complex is a pair $(C,\iota_K)$ which satisfies the following properties.
\begin{itemize}
    \item $C$ is a chain complex of finitely generated free modules over $\mathbb{F}_2[U,V]$, such that $(U,V)^{-1}C$ has homology $(U,V)^{-1}\mathbb{F}_2 [U,V]$.
    \item $\iota_K$ is a homotopy skew-autoequivalence of $C$ such that $\iota_K ^2 \sim 1+\Phi\Psi$, where $\Phi$ and $\Psi$ are the formal derivatives of the differential $\partial$ of $C$ with respect to the formal variables $U$ and $V$, respectively.
\end{itemize}
Given two chain complexes $C_1$ and $C_2$ of free modules over $\mathbb{F}_2[U,V]$, a chain map $f:C_1\rightarrow C_2$ is said to be a \emph{local map} if the maps
\[
\begin{split}
    U^{-1}f\vert_{V=0} &: C\otimes \mathbb{F}_2[U^{\pm 1}] \rightarrow D\otimes \mathbb{F}[U^{\pm 1}], \\
    V^{-1}f\vert_{U=0} &: C\otimes \mathbb{F}_2[V^{\pm 1}] \rightarrow D\otimes \mathbb{F}[V^{\pm 1}]
\end{split}
\]
induce injective maps in homology. Given two $\iota_K$-complexes $(C_1,\iota_1)$ and $(C_2,\iota_2)$, a local map $f:C_1 \rightarrow C_2$ is said to be a \emph{$\iota_K$-local map} if $\iota_2 \circ f \sim f \circ \iota_1$. If $\iota_K$-local maps between two $\iota_K$-complexes exist in both directions, then we say that they are \emph{$\iota_K$-locally equivalent}. The set of $\iota_K$-local equivalence classes of $\iota_K$-complexes form a group $\mathfrak{I}_K$ when endowed with the addition operation
\[
\iota_{C_1\otimes C_2} = (\mathbf{id}+\Phi\otimes \Psi)\circ (\iota_1 \otimes \iota_2).
\]
As in the 3-manifold case, the construction of involutive knot Floer homology gives a canonical map $\mathcal{C}\rightarrow \mathfrak{I}_K$.

We will sometimes work with knot Floer homology with coefficient ring $\mathbb{F}_2 [U,V]/(UV)$, which is denoted as $\mathcal{R}$, rather than the full two-variable ring $\mathbb{F}_2[U,V]$. Note that although $\iota_K$-local maps and $\iota_K$-local equivalences are well-defined, it is unclear whether $\iota_K$-local equivalence classes of involutive $\mathcal{R}$-coefficient knot Floer chain complexes form a well-defined group, since the basepoint actions might not be uniquely determined from the $\mathcal{R}$-coefficient differential.

\subsection{Involutive bordered Floer homology}
Let $M$ be a bordered 3-manifold with one boundary; for simplicity, we will assume that $\partial M$ is a torus. Choose a bordered Heegaard diagram $H=(\Sigma,\boldsymbol\alpha,\boldsymbol\beta,z)$ representing $M$ and consider its conjugate diagram $\bar{H}=(-\Sigma,\boldsymbol\beta,\boldsymbol\alpha,z)$. Then we have canonical identification maps
\[
\begin{split}
    \mathbf{conj}:\widehat{CFD}(H)\rightarrow \widehat{CFD}(\bar{H}),\\
    \mathbf{conj}:\widehat{CFA}(H)\rightarrow \widehat{CFA}(\bar{H}).
\end{split}
\]
between the type-D and type-A modules associated to $H$ and $\bar{H}$, respectively. Note that we are using the same name for the type-D and type-A identification maps for convenience.

In contrast to the case of closed 3-manifolds, there does not exist a sequence of Heegaard moves from $\bar{H}$ to $H$. The reason is that $H$ is $\alpha$-bordered, whereas $\bar{H}$ is $\beta$-bordered. To remedy this problem, Hendricks and Lipshitz \cite{hendricks2019involutivebordered} uses the Auroux-Zarev piece $\mathbf{AZ}$ and its conjugate $\overline{\mathbf{AZ}}$, which satisfies the property that $\overline{\mathbf{AZ}}\cup \mathbf{AZ}$ represents a trivial cylinder $T^2 \times I$. A Heegaard diagram representing $\mathbf{AZ}$ is shown in \Cref{fig:AZpiece}.

\begin{figure}[hbt]
    \centering
    \includegraphics[width=0.3\textwidth]{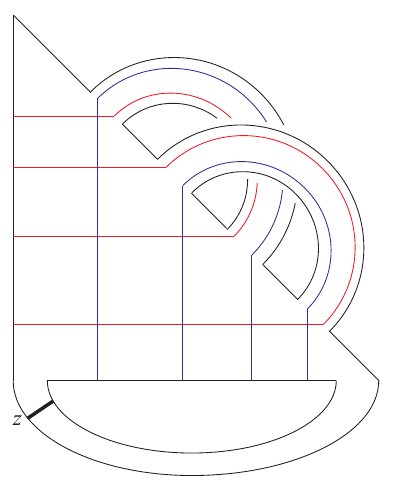}
    \caption{A $\alpha$-$\beta$-bordered Heegaard diagram representing the Auroux-Zarev piece $\mathbf{AZ}$. Note that this diagram is nice, in the sense of \cite[Definition 8.1]{lipshitz2018bordered}.}
\label{fig:AZpiece}
\end{figure}

    One starts with the \cite[Theorem 4.6]{lipshitz2011heegaard}, which implies that $\mathbf{AZ} \cup \bar{H}$ and $\bar{H}\cup \overline{\mathbf{AZ}}$ are related to $H$ by a sequence of Heegaard moves. Choosing such sequences give homotopy equivalences
\[
\begin{split}
    \Phi_{\mathbf{AZ}\cup \bar{H},H} &:\widehat{CFD}(\mathbf{AZ}\cup \bar{H})\rightarrow \widehat{CFD}(H),\\
    \Phi_{\bar{H}\cup \overline{\mathbf{AZ}},H} &: \widehat{CFA}(\bar{H} \cup \overline{\mathbf{AZ}}) \rightarrow \widehat{CFA}(H).
\end{split}
\]
Recall that we have pairing maps induced by time dilation, as discussed in \cite[Chapter 9]{lipshitz2018bordered}, which are defined uniquely up to homotopy:
\[
\begin{split}
    \widehat{CFD}(\mathbf{AZ}\cup \bar{H}) &\xrightarrow{\simeq} \widehat{CFDA}(\mathbf{AZ})\boxtimes \widehat{CFD}(\bar{H}), \\
    \widehat{CFA}(\bar{H}\cup \overline{\mathbf{AZ}}) &\xrightarrow{\simeq} \widehat{CFA}(\bar{H})\boxtimes \widehat{CFDA}(\overline{\mathbf{AZ}}).
\end{split}
\]
Then we can define the bordered involution $\iota_M$, in both type-D and type-A modules, as follows:
\[
\begin{split}
    \iota_M &: \widehat{CFDA}(\mathbf{AZ})\boxtimes \widehat{CFD}(H)\xrightarrow{\mathbf{id}\boxtimes \mathbf{conj}} \widehat{CFDA}(\mathbf{AZ})\boxtimes \widehat{CFD}(\bar{H})\xrightarrow{\simeq} \widehat{CFD}(\mathbf{AZ}\cup \bar{H})\xrightarrow{\Phi_{\mathbf{AZ}\cup \bar{H},H}} \widehat{CFD}(H), \\
    \iota_M &: \widehat{CFA}(H)\boxtimes \widehat{CFDA}(\overline{\mathbf{AZ}}) \xrightarrow{\mathbf{conj}\boxtimes \mathbf{id}} \widehat{CFD}(\bar{H})\boxtimes \widehat{CFDA}(\overline{\mathbf{AZ}})\xrightarrow{\simeq} \widehat{CFA}(\bar{H}\cup\overline{\mathbf{AZ}})\xrightarrow{\Phi_{\bar{H}\cup \overline{\mathbf{AZ}},H}} \widehat{CFA}(H).
\end{split}
\]

Now suppose that we are given a bordered 3-manifold $N$ whose boundary consists of two torus components. Choose an $\alpha$-$\alpha$-bordered Heegaard diagram $H$ representing $N$. Then it follows again from \cite[Theorem 4.6]{lipshitz2011heegaard} that $\mathbf{AZ}\cup \bar{H}\cup \overline{\mathbf{AZ}}$ is related to $H$ by a sequence of Heegaard moves. Choosing such a sequence gives a homotopy equivalence 
\[
\Phi_{\mathbf{AZ}\cup \bar{H}\cup \overline{\mathbf{AZ}},H}: \widehat{CFDA}(\mathbf{AZ} \cup \bar{H} \cup \overline{\mathbf{AZ}}) \rightarrow \widehat{CFDA}(H).
\]
Thus we can define a bordered involution $\iota_N$ as follows.
\[
\begin{split}
    \iota_N:\widehat{CFDA}(\mathbf{AZ})\boxtimes \widehat{CFDA}(H)\boxtimes \widehat{CFDA}(\overline{\mathbf{AZ}}) &\xrightarrow{\mathbf{id}\boxtimes \mathbf{conj}\boxtimes \mathbf{id}} \widehat{CFDA}(\mathbf{AZ})\boxtimes \widehat{CFDA}(\bar{H})\boxtimes \widehat{CFDA}(\overline{\mathbf{AZ}}) \\
    &\xrightarrow{\simeq} \widehat{CFDA}(\mathbf{AZ} \cup \bar{H} \cup \overline{\mathbf{AZ}}) \\
    &\xrightarrow{\Phi_{\mathbf{AZ}\cup \bar{H}\cup \overline{\mathbf{AZ}},H}} \widehat{CFDA}(H).
\end{split}
\]
Unlike the cases of knots and closed 3-manifolds, we do not know whether the homotopy classes of $\iota_M$ and $\iota_N$ are independent of our choices of sequences of Heegaard moves. This is because a naturality result for bordered Heegaard Floer homology is currently unknown. However, we can instead consider the sets of all possible involutions coming from any possible choices of sequences of Heegaard moves, as shown in the definition below.

\begin{defn}
Given a bordered 3-manifold $M$ with one torus boundary, we denote the set of all possible involutions 
\[
\begin{split}
\iota_M &: \widehat{CFDA}(\mathbf{AZ})\boxtimes \widehat{CFD}(M)\rightarrow\widehat{CFD}(M), \\
\iota_M &: \widehat{CFA}(M) \boxtimes \widehat{CFDA}(\overline{\mathbf{AZ}}) \rightarrow \widehat{CFD}(M),
\end{split}
\]
induced by choosing a sequence of Heegaard moves from $\mathbf{AZ}\cup \bar{H}$ and $\bar{H}\cup \overline{\mathbf{AZ}}$ to $H$ as $\mathbf{Inv}_D(M)$ and $\mathbf{Inv}_A(M)$, respectively. Furthermore, given a bordered 3-manifold $N$ with two torus boundaries, we similarly denote the set of all possible involutions $$\iota_M:\widehat{CFDA}(\mathbf{AZ})\boxtimes \widehat{CFDA}(N)\boxtimes \widehat{CFDA}(\overline{\mathbf{AZ}})\rightarrow\widehat{CFDA}(N)$$ induced by choosing a sequence of Heegaard moves from $\mathbf{AZ} \cup \bar{H}\cup \overline{\mathbf{AZ}}$ to $H$ as $\mathbf{Inv}(N)$.
\end{defn}

Recall that, given two bordered 3-manifolds $M_1$ and $M_2$, we have a pairing theorem 
\begin{equation}
    \widehat{CF}(M_1 \cup M_2) \simeq \widehat{CFA}(M_1) \boxtimes \widehat{CFD}(M_2).
    \label{eqn:boxpairing}
\end{equation}
Due to the pairing theorem for triangles \cite[Proposition 5.35]{lipshitz2016bordered}, it is clear that the homotopy equivalence used in \Cref{eqn:boxpairing} is well-defined up to homotopy. \cite[Theorem 5.1]{hendricks2019involutivebordered} tells us that for any $\iota_1 \in \mathbf{Inv}_D(M_1)$ and $\iota_2 \in \mathbf{Inv}_D (M_2)$, the map 
\[
\begin{split}
\widehat{CF}(M_1 \cup M_2) &\xrightarrow{\text{pairing}} \widehat{CFA}(M_1) \boxtimes \widehat{CFD}(M_2) \\
&\xrightarrow{\simeq} \widehat{CFA}(M_1)\boxtimes \widehat{CFDA}(\overline{\mathbf{AZ}})\boxtimes \widehat{CFDA}(\mathbf{AZ})\boxtimes \widehat{CFD}(M_2) \\
&\xrightarrow{\iota_1 \boxtimes \iota_2} \widehat{CFA}(M_1) \boxtimes \widehat{CFD}(M_2) \xrightarrow{\text{pairing}} \widehat{CF}(M_1 \cup M_2)
\end{split}
\]
is homotopic to the involution $\iota_{M_1 \cup M_2}$ on $\widehat{CF}(M_1 \cup M_2)$.

One also has another pairing formula involving morphism spaces between type-D modules. Given two bordered 3-manifolds $M_1$ and $M_2$ with one torus boundary, one can also obtain the hat-flavored Heegaard Floer homology of $-M_1 \cup M_2$ as follows\cite[Theorem 1]{lipshitz2011heegaard}:
\begin{equation}
\widehat{CF}(-M_1 \cup M_2) \simeq \mathbf{Mor}(\widehat{CFD}(M_1),\widehat{CFD}(M_2)),
\label{eqn:morpairing}
\end{equation}
Unlike the box tensor product version of pairing formula, the well-definedness of homotopy equivalence up to homotopy in the above formula is not entirely obvious. This is because its proof relies on the following isomorphism:
\[
\begin{split}
\widehat{CF}(-M_1 \cup M_2 ) &\xrightarrow{\text{pairing}} \widehat{CFA}(-M_1) \boxtimes \widehat{CFD}(M_2)\\
&\xrightarrow{M_1 \simeq M_1 \cup \mathbf{AZ}} \overline{\widehat{CFD}(-M_1)} \boxtimes \widehat{CFAA}(\mathbf{AZ}) \boxtimes \widehat{CFD}(M_2) \\
&\xrightarrow{\text{canonical isomorphism}} \mathbf{Mor}(\widehat{CFD}(M_1),\widehat{CFD}(M_2)).
\end{split}
\]
In particular, the homotopy equivalence $\widehat{CFA}(M_1)\simeq \widehat{CFD}(M_1)\boxtimes \widehat{CFAA}(\mathbf{AZ})$, which is induced by a sequence of Heegaard moves from $M_1$ to $M_1 \cup \mathbf{AZ}$, may not be well-defined due to the lack of naturality. However, if we have two such sequences which induce two identification maps
\[
F,G:\widehat{CF}(-M_1 \cup M_2) \simeq \mathbf{Mor}(\widehat{CFD}(M_1),\widehat{CFD}(M_2)),
\]
then by the pairing theorem for triangles, the map $G^{-1}\circ F:\widehat{CF}(-M_1\cup M_2)\rightarrow \widehat{CF}(-M_1 \cup M_2)$ is the homotopy autoequivalence induced by a loop of Heegaard moves, which should be homotopic to identity due to naturality. Therefore the homotopy equivalence used in \Cref{eqn:morpairing} is well-defined up to homotopy.

Now it follows from the proof of \cite[Theorem 8.5]{hendricks2019involutivebordered} that the map 
\[
\begin{split}
\widehat{CF}(-M_1 \cup M_2) &\xrightarrow{\simeq} \mathbf{Mor}(\widehat{CFD}(M_1),\widehat{CFD}(M_2)) \\
&\xrightarrow{f\mapsto \mathbf{id}\boxtimes f} \mathbf{Mor}(\widehat{CFDA}(\mathbf{AZ})\boxtimes \widehat{CFD}(M_1),\widehat{CFDA}(\mathbf{AZ})\boxtimes \widehat{CFD}(M_2)) \\
&\xrightarrow {g\mapsto \iota_2 \circ g \circ \iota_1 ^{-1}} \mathbf{Mor}(\widehat{CFD}(M_1),\widehat{CFD}(M_2)) \\
&\xrightarrow{\simeq} \widehat{CF}(-M_1 \cup M_2)
\end{split}
\]
is homotopic to the involution $\iota_{-M_1 \cup M_2}$ on $\widehat{CF}(-M_1 \cup M_2)$ for any choice of $\iota_1 \in \mathbf{Inv}_D(M_1)$ and $\iota_2 \in \mathbf{Inv}_D(M_2)$.

\section{Involutive knot Floer homology with a free basepoint}
\label{sec:freebasepoint}

Given a knot $K$, instead of choosing a doubly-pointed Heegaard diagram representing $K$, we consider a multipointed Heegaard diagram $H=(\Sigma,\boldsymbol\alpha,\boldsymbol\beta,\{z,z_{free}\},w)$, where $z$ and $w$ are points on $K$ and $z_{free}$ is a \emph{free basepoint}, which lies outside $K$. Given such a diagram, we define its 2-variable knot Floer homology 
\[
CFK_{UV}(S^3,K,z_{free})=CF_{UV}(H) \coloneqq H_{\ast} \left( \bigoplus_{\mathbf{x}\in \mathbb{T}_{\boldsymbol\alpha}\cap \mathbb{T}_{\boldsymbol\beta}} \mathbb{F}_2[U,V]\mathbf{x},\partial \right),
\]
where the differential $\partial$ is defined using the formula 
\[
\partial \mathbf{x}=\sum_{\mathbf{y}\in \mathbb{T}_{\boldsymbol\alpha}\cap \mathbb{T}_{\boldsymbol\beta}} \sum_{\substack{\phi\in \pi_2(\mathbf{x},\mathbf{y}) \\ n_{z_{free}}(\phi)=0 \\ \mu(\phi)=1}} \sharp \widetilde{\mathcal{M}(\phi)} \cdot U^{n_{z}(\phi)} V^{n_{w}(\phi)} \mathbf{y}.
\]
Here, $\mathcal{M}(\phi)$ denotes the moduli space of holomorphic curves representing the given homotopy class $\phi$ of Whitney disks from $\mathbf{x}$ to $\mathbf{y}$, and $n_{z_{free}}(\phi)$, $n_z(\phi)$, and $n_w(\phi)$ denote the algebraic intersection number of $\phi$ with the codimension 2 submanifolds given by $z_{free}$, $z$, and $w$, respectively. Note that the naturality result for Heegaard Floer homology \cite{juhasz2012naturality} also applies to this case, so that chain homotopy autoequivalences of $CFK_{UV}(S^3,K,z_{free})$ induced by any loop of Heegaard moves connecting Heegaard diagrams representing $(S^3,K,z_{free})$ are homotopic to the identity map.

As in involutive knot Floer homology, we can define the conjugate diagram $\bar{H}$ of $H$ as follows:
\[
\bar{H}=(-\Sigma,\boldsymbol\beta,\boldsymbol\alpha,\{w,z_{free}\},z).
\]
We have a canonically defined chain skew-isomorphism:
\[
\mathbf{conj}: CF_{UV}(H)\rightarrow CF_{UV}(\bar{H}).
\]
We then consider the half-twist self-diffeomorphism $\phi$ of $(S^3,K,z_{free})$ which maps $z$ and $w$ to $w$ and $z$, respectively. It induces a diffeomorphism map 
\[
\phi_{\ast}:CF_{UV}(\bar{H})\rightarrow CF_{UV}(\phi(\bar{H})).
\]
Then, since $\phi(\bar{H})$ and $H$ both represent $(S^3,K,z_{free})$, there exists a sequence of Heegaard moves between them, which induces a homotopy equivalence 
\[
\Phi_{\phi(\bar{H}),H}:CF_{UV}(\phi(\bar{H}))\rightarrow CF_{UV}(H),
\]
which is well-defined up to chain homotopy, due to naturality. Composing the above three maps thus gives 
\[
\iota_{K,z_{free}}=\Phi_{\phi(\bar{H}),H}\circ \phi_{\ast} \circ \mathbf{conj},
\]
which is again well-defined up to chain homotopy.

Given a doubly-pointed Heegaard diagram $H_K$ representing $K$, we can perform a free-stabilization on $H_K$ near the basepoint $z$, as shown in the left of \Cref{fig:freestab}, to get a new diagram $H^{st}_K$ representing $(S^3,K,z_{free})$. Then, by \cite[Lemma 7.1]{zemke2019link}, the differential of $CF_{UV}(H^{st}_K)$ is given by the matrix 
\[
\partial_{CF_{UV}(H^{st}_K)} = \begin{pmatrix}
\partial_{CF_{UV}(H_K)} & UV \\ 0 & \partial_{CF_{UV}(H_K)} 
\end{pmatrix},
\]
where we are using an identification 
\[
CF_{UV}(H^{st}_K)\simeq CF_{UV}(H_K)\otimes \left( \mathbb{F}_2 \cdot \theta^+ \oplus \mathbb{F}_2 \cdot \theta^- \right)
\]
of the chain group. Furthermore, the free-stabilization map $S_{z_{free}}:CF_{UV}(H_K)\rightarrow CF_{UV}(H^{st}_K)$, defined as 
\[
S^+_{z_{free}}(\mathbf{x})=\mathbf{x}\otimes \theta^{+},
\]
depends only on the isotopy class of $K$.

\begin{figure}[hbt]
    \centering
    \includegraphics[width=0.7\textwidth]{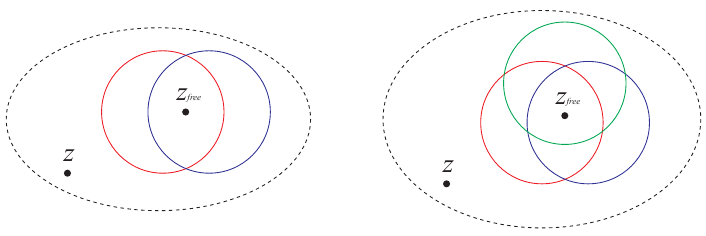}
    \caption{Left, a free-stabilization of a Heegaard diagram near a basepoint $z$. Right, a free-stabilization of a Heegaard triple-diagram near the same basepoint $z$.}
\label{fig:freestab}
\end{figure}

We now assume that $K$ is boundary-parallel to the Heegaard surface $\Sigma$ and the self-diffeomorphism $\phi$ acts as identity near the free-stabilization locus. Then $\phi(\bar{H}^{st}_K)$ is also a free-stabilization on $\phi(\bar{H}_K)$ near the basepoint $z$, and for any sequence $H_0=\phi(\bar{H}_K)\rightarrow H_1 \rightarrow \cdots \rightarrow H_n=H_K$ of Heegaard diagrams, we have a corresponding sequence $H^{st}_0=\phi(\bar{H}^{st}_K)\rightarrow\cdots\rightarrow H^{st}_n=H^{st}_K$ such that for each $i$, $H^{st}_i$ is a free-stabilization of $H_i$ near $z$. 

For each $i$, the Heegaard move $H_i\rightarrow H_{i+1}$ is either an isotopy, a handleslide, or a stabilization. Since we can always start with sufficiently stabilized diagrams and replace an isotopy by a sequence of handleslides, we may further assume that all Heegaard moves that we use are handleslides. Recall that the chain homotopy equivalences associated to handleslides are defined by counting holomorphic triangles in a Heegaard triple diagram. If the homotopy equivalence $f_i:CF_{UV}(H_i)\rightarrow CF_{UV}(H_{i+1})$ is defined by counting triangles in a triple diagram $H_T=(\Sigma,\boldsymbol\alpha,\boldsymbol\beta,\boldsymbol\gamma,z,w)$, then the homotopy equivalence $f^{st}_i:CF_{UV}(H^{st}_i)\rightarrow CF_{UV}(H^{st}_{i+1})$ is defined by counting triangles in a triple diagram $H^{st}_T$ which is obtained by free-stabilizing $H_T$ near $z$, as shown in the right of \Cref{fig:freestab}. Thus, by \cite[Theorem 6.7]{zemke2015graph}, we know that 
\[
S^+_{z_{free}} \circ f_i \sim f^{st}_i \circ S^+_{z_{free}},
\]
so we deduce that $S^+_{z_{free}}$ is well-defined up to homotopy and 
\[
S^+_{z_{free}} \circ \iota_k \sim \iota_{K,z_{free}} \circ S^+_{z_{free}}.
\]
Furthermore, since the truncated map $S^+_{z_{free}}\vert_{U=1,V=0}$ is the hat-flavored  free-stabilzation map 
\[
S^+_{z_{free}}:\widehat{CFK}(S^3,z)\rightarrow \widehat{CFK}(S^3,\{z,z_{free}\}),
\]
which is injective, and tensoring it with $\mathbb{F}_2[U^{\pm 1}]$ gives $U^{-1}S^+_{z_{free}}\vert_{V=0}$, we see that $U^{-1}S^+_{z_{free}}\vert_{V=0}$ (and also $V^{-1}S^+_{z_{free}}\vert_{U=0}$) induces an injective map in homology. Therefore $S^+_{z_{free}}$ is local.

We now interpret involutive knot Floer theory with a free basepoint in terms of bordered Floer homology. Consider the triply-pointed bordered Heegaard diagram $\X=(\Sigma,\boldsymbol\alpha,\boldsymbol\beta,\{z,z_{free}\},w)$, defined as in \Cref{fig:diagramX}. This diagram represents the longitudinal knot lying inside the $\infty$-framed solid torus, together with a prescribed free basepoint $z_{free}$ on the boundary torus.

\begin{figure}[hbt]
    \centering
    \includegraphics[width=0.4\textwidth]{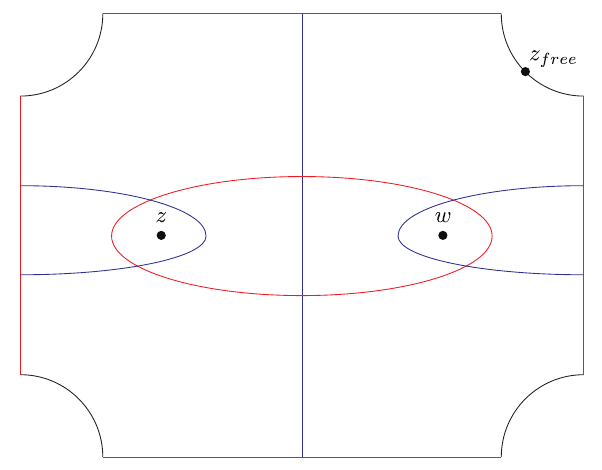}
    \caption{The triply-pointed Heegaard diagram $\X$.}
\label{fig:diagramX}
\end{figure}

Note that, for any bordered Heegaard diagram $H$ of $M \backslash K$, where $K$ is a framed knot inside a closed 3-manifold $M$ and the framing is denoted as $\nu$, the glued diagram $H\cup \X$ is a Heegaard diagram representing the core curve inside the Dehn surgery $M_\nu (K)$, together with a free basepoint.

We now consider the new diagram $\phi(\bar{\X})$, where $\bar{\X}$ denotes the conjugate diagram of $\X$, defined as 
\[
\bar{\X}=(-\Sigma,\boldsymbol\beta,\boldsymbol\alpha,\{w,z_{free}\},z),
\]
and $\phi$ denotes the ``half-twist'' self-diffeomorphism of $\Sigma$ along the longitudinal knot, so that it maps $z$ to $w$ and $w$ to $z$, respectively.

\begin{lem}
\label{lemma:Xduality}
Consider the $\alpha$-$\beta$-bordered Auroux-Zarev piece $\mathbf{AZ}$. Then $\overline{\mathbf{AZ}}\cup \phi(\bar{\X})$ and $\X$ are related by a sequence of Heegaard moves.
\end{lem}
\begin{proof}
Denote the bordered Heegaard diagram representing the 0-framed solid torus as $H$, and its conjugate as $\bar{H}$. It is proven in \cite[Figure 8 and 9]{lipshitz2011heegaard} that $\overline{\mathbf{AZ}}\cup \bar{H}$ and $H$ are related by a sequence of handleslides and a destabilization. Since $H$ is simply $\X$ without the basepoints $z$, $w$, and the $\alpha$- and $\beta$-curves surrounding them, it is clear that the sequence of handleslides (and a single destabilization) from $\overline{\mathbf{AZ}}\cup \bar{H}$ to $H$ induces sequence of handleslides and a single destabilization from $\overline{\mathbf{AZ}}\cup \phi(\bar{\X})$ to $\X$. A detailed process is drawn in \Cref{fig:HeegaardmovesX}.
\end{proof}

\begin{figure}[hbt]
    \centering
    \includegraphics[width=0.8\textwidth]{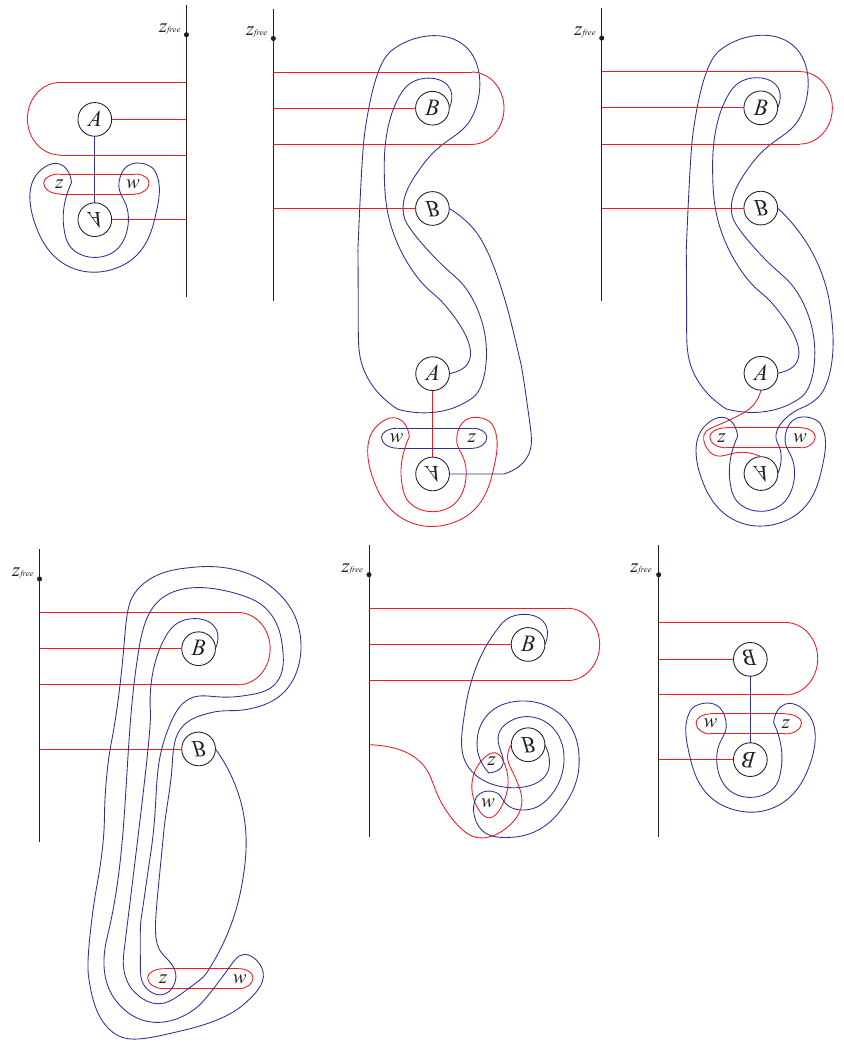}
    \caption{Top-left, the diagram $\X$. Top-middle, the diagram $\overline{\mathbf{AZ}}\cup\bar{\X}$. Top-right, the diagram $\overline{\mathbf{AZ}}\cup \phi(\bar{\X})$. Bottom-left, A diagram obtained from the one on the top-right by a sequence of handleslides, followed by a destabilization. Bottom-middle, A diagram obtained from the one on the bottom-left by another sequence of handleslides. Bottom-right, the diagram obtained by isotopy from the one on the bottom-middle. Note that this is the same as the original diagram $\X$.}
\label{fig:HeegaardmovesX}
\end{figure}

Choose a nice diagram $\X_0$ which is related by $\X$ by a sequence of Heegaard moves; such a diagram always exists by Sarkar-Wang algorithm \cite[Proposition 8.2]{lipshitz2018bordered}, and it is always provincially admissible. Then $\X_0$ has a well-defined bordered Floer homology. In particular, if we write $\X_0 = (\Sigma,\boldsymbol\alpha,\boldsymbol\beta,\{z,z_{free}\},w)$, then we have a well-defined type-D structure $CFD_{UV}(\X_0)$ and a type-A structure $CFA_{UV}(\X_0)$ over the module $\mathbb{F}_2[U,V]$, defined by counting holomorphic disks which do not intersect algebraically with $z_{free}$, while recording their algebraic intersection numbers with $z$ and $w$ by formal variables $U$ and $V$, respectively.

Recall from \cite[Chapter 10]{lipshitz2018bordered} that, given a bordered 3-manifold $Y$ with boundary $\mathcal{Z}$, the associated type-A module $\widehat{CFA}(Y)$ is graded by a transitive $G(\mathcal{Z})$-set, and for a doubly-pointed bordered Heegaard diagram $H=(\Sigma,\boldsymbol\alpha,\boldsymbol\beta,z,w)$ with the same boundary, the associated type-D module $\widehat{CFA}^-(H)$ admits an enhanced grading by a transitive $(G(\mathcal{Z})\times \mathbb{Z})$-set, where the grading on the $\mathbb{Z}$ component is given by $n_w-n_z$. We can define a grading on $CFA_{UV}(\X_0)$ by the group $G(T^2)\times \mathbb{Z}$ is a similar manner, as follows.

Write $\X_0=(\Sigma,\boldsymbol\alpha,\boldsymbol\beta,\{z,z_{free}\},w)$. Then for any choice of Floer generators $\mathbf{x}$ and $\mathbf{y}$ and a homology class $B\in \pi_2(\mathbf{x},\mathbf{y})$ of curves connecting $\mathbf{x}$ to $\mathbf{y}$, we define the relative grading $\tilde{g}(\mathbf{x},\mathbf{y})\in G(T^2)\times\mathbb{Z}$ as
\begin{equation}
\label{eqn:grading}
\tilde{g}(\mathbf{x},\mathbf{y})= (\lambda^{-2n_{z_{free}}(B)}g(\mathbf{x},\mathbf{y}),n_z(B)-n_w(B)),    
\end{equation}

where $\lambda$ is the central element $(1;0,0)$ of $G(T^2)$ and $g$ denotes the quantity determined by \cite[Formula 10.31]{lipshitz2018bordered}. This endows $CFA_{UV}(\X_0)$ with a grading by a transitive $(G(T^2)\times \mathbb{Z})$-set. After taking a box tensor product with $\widehat{CFD}(S^3 \backslash K)$, where $K$ is a knot, the gradings on $\widehat{CFD}(S^3 \backslash K)$ and $CFA_{UV}(\X_0)$ induce a grading on the tensor product.

\begin{lem}
\label{lem:Xpairing}
Given a knot $K\subset S^3$, denote the bordered 3-manifold representing its 0-framed complement as $S^3 \backslash K$. Then we have a pairing formula
\[
\widehat{CFD}(S^3 \backslash K) \boxtimes CFA_{UV}(\X_0) \simeq CFK_{UV}(S^3,K,z_{free}).
\]
Furthermore, the induced grading on the left hand side matches the bigrading (i.e. Maslov and Alexander) on the right hand side.
\end{lem}
\begin{proof}
Choose a nice bordered Heegaard diagram $\mathcal{H}$ representing $S^3 \backslash K$. Since the proof of pairing theorem \cite[Theorem 1.3]{lipshitz2018bordered} works trivially for admissible diagrams, we have 
\[
\widehat{CFD}(S^3 \backslash K) \boxtimes CFA_{UV}(\X_0) \simeq CF_{UV}(\mathcal{H}\cup \X_0).
\]
The Heegaard diagram $\mathcal{H}\cup \X_0$ represents $K$, together with a free basepoint $z_{free}$ lying outside $K$, we get the desired homotopy equivalence. The statement about gradings follows directly from the arguments used in the proof of \cite[Theorem 1.3]{lipshitz2018bordered}.
\end{proof}

\begin{rem}
In the proof of \Cref{lem:Xpairing}, the term $CF_{UV}(\mathcal{H}\cup \X_0)$ is the Floer chain complex coming from cylindrical reformulation of Heegaard Floer homology, due to Lipshitz\cite{lipshitz2006cylindrical}. The original setting of cylindrical reformation is only for Heegaard diagrams with one basepoint, so it is natural to ask whether it also works for general diagrams $(\Sigma,\boldsymbol\alpha,\boldsymbol\beta,\mathbf{z})$, where the number of $\alpha$-curves may exceed the genus of $\Sigma$ (in which case we have more than one basepoints). Fortunately, the cylindrical reformation also works in those generalized settings; see \cite[Section 5.2]{ozsvath2008holomorphic} for details.
\end{rem}

\begin{lem}
There exist type-D homotopy equivalences:
\[
\begin{split}
\label{lem:truncationlem}
CFA_{UV}(\X_0) \otimes \mathbb{F}[U,V]/(U-1,V) &\simeq CFA_{UV}(\X_0) \otimes \mathbb{F}[U,V]/(U,V-1) \\
&\simeq \widehat{CFA}(0\text{-framed solid torus})\otimes \mathbb{F}^2 _2.
\end{split}
\]
\end{lem}
\begin{proof}
Write $\X_0=(\Sigma,\boldsymbol\alpha,\boldsymbol\beta,\{z,z_{free}\},w)$. Since truncating by $V=1$ is equivalent to forgetting the basepoint $w$, we have a following homotopy equivalence of type-D modules:
\[
CFA_{UV}(\X_0)\otimes \mathbb{F}[U,V]/(U,V-1) \simeq \widehat{CFA}(\Sigma,\boldsymbol\alpha,\boldsymbol\beta,\{z,z_{free}\}).
\]
Since we no longer have $w$ as a basepoint, the bordered Heegaard diagram $(\Sigma,\boldsymbol\alpha,\boldsymbol\beta,\{z,z_{free}\})$ is isotopic to the diagram we obtain by stabilizing a bordered Heegaard diagram representing the 0-framed solid torus near its basepoint. Since we are not counting holomorphic disks intersecting the stabilization region, it is clear, even without a neck-stretching argument, that we have a canonical isomorphism 
\[
\widehat{CFA}(\Sigma,\boldsymbol\alpha,\boldsymbol\beta,\{z,z_{free}\}) \simeq \widehat{CFA}(0\text{-framed solid torus})\otimes \mathbb{F}^2_2,
\]
which proves the lemma.
\end{proof}

Let $\bar{\X}_0$ be the conjugate diagram of $\X_0$, defined in the same way as $\bar{\X}$. Then, by \Cref{lemma:Xduality}, we know that $\overline{\mathbf{AZ}}\cup \phi(\bar{\X}_0)$ is related by a sequence of Heegaard moves to $\X_0$. As in the proof of \Cref{lem:Xpairing}, it is clear that we have a pairing formula 
\[
CFA_{UV}(\bar{\X}_0)\boxtimes \widehat{CFDA}(\overline{\mathbf{AZ}})\simeq CFA_{UV}(\bar{\X}_0\cup \overline{\mathbf{AZ}}),
\]
so any choice of a sequence of Heegaard moves from $\overline{\mathbf{AZ}}\cup \phi(\bar{\X}_0)$ to $\X_0$ induces a type-A morphism
\[
\iota_{\X}\,:\, CFA_{UV}(\phi(\bar{\X}_0)) \boxtimes \widehat{CFDA}(\overline{\mathbf{AZ}}) \rightarrow CFA_{UV}(\X_0).
\]
Note that $\iota_{\X}$ is a homotopy equivalence of type-A modules over $\mathbb{F}_2$, but not over $\mathbb{F}_2[U,V]$; this is because it intertwines the actions of $U$ and $V$. Thus $\iota_{\X}$ is a type-A homotopy \emph{skew-equivalence}. 

The definition of $\iota_{\X}$ depends on the choices that we have made in its construction. Choosing a different sequence of Heegaard moves may result in another homotopy equivalence which is not homotopic to $\iota_{\X}$, due to the lack of naturality for bordered Floer homology. However it will not affect the results of this paper; we only have to choose one sequence of Heegaard moves, once and for all.

Given a knot $K\subset S^3$ and a bordered Heegaard diagram $\mathcal{H}$ for the 0-framed complement of $K$, recall that we can choose a homotopy equivalence
\[
\iota_{S^3 \backslash K}\,:\,\widehat{CFD}(\bar{\mathcal{H}}\cup \overline{\mathbf{AZ}}) \rightarrow \widehat{CFD}(\mathcal{H}),
\]
which is an element of $\mathbf{Inv}_D(S^3 \backslash K)$. Furthermore, we have the following conjugation maps:
\[
\begin{split}
    \mathbf{conj}_{\X} &: CFA_{UV}(\X_0) \rightarrow CFA_{UV}(\bar{\X}_0), \\
    \mathbf{conj}_{S^3 \backslash K} &: CFD(\mathcal{H}) \rightarrow CFD(\bar{\mathcal{H}}).
\end{split}
\]
We consider the following composition of homotopy equivalences, which we will denote as $F_K$.
\[
\begin{split}
   F_K: CFK_{UV}(S^3,K,z_{free}) &\simeq CFA_{UV}(\X_0)\boxtimes \widehat{CFD}(\mathcal{H}) \\
    &\xrightarrow{\mathbf{conj}_{\X} \boxtimes \mathbf{conj}_{S^3 \backslash K}} CFA_{UV}(\bar{\X}_0) \boxtimes \widehat{CFD}(\bar{\mathcal{H}})\\
    &\xrightarrow{\phi_{\ast}\boxtimes \mathbf{id}} CFA_{UV}(\phi(\bar{\X}_0)) \boxtimes \widehat{CFD}(\bar{\mathcal{H}})\\
    &\simeq CFA_{UV}(\phi(\bar{\X}_0))\boxtimes \widehat{CFDA}(\overline{\mathbf{AZ}}) \boxtimes \widehat{CFDA}(\mathbf{AZ}) \boxtimes \widehat{CFD}(\bar{\mathcal{H}}) \\
    &\simeq CFA_{UV}(\phi(\bar{\X}_0) \cup \overline{\mathbf{AZ}}) \boxtimes \widehat{CFDA}(\mathbf{AZ}) \boxtimes \widehat{CFD}(\bar{\mathcal{H}}) \\
    &\xrightarrow{\iota_{\X}\boxtimes \iota_{S^3\backslash K}} CFA_{UV}(\X_0)\boxtimes \widehat{CFD}(\mathcal{H}) \simeq CFK_{UV}(S^3,K,z_{free})
\end{split}
\]

\begin{lem}
\label{lem:freestabinv}
For any choice of $\iota_{S^3 \backslash K}\in \mathbf{Inv}_D(S^3 \backslash K)$, the induced homotopy equivalence $F_K$ is homotopic to $\iota_{K,z_{free}}$.
\end{lem}
\begin{proof}
One can use the argument used in the proof of \cite[Theorem 5.1]{hendricks2019involutivebordered} verbatim.
\end{proof}

For later use, we prove the following lemma.

\begin{lem}
\label{lem:eliminationlem}
Given a knot $K$, suppose that there exists a local chain map 
\[f:CFK_{UV}(S^3,K,z_{free}) \rightarrow CFK_{UV}(S^3,\mathbf{unknot},z_{free})
\]
which preserves the Alexander and Maslov gradings, such that $f\circ \iota_{K,z_{free}}\sim f$. Then there also exists a local (bidegree-preserving) chain map $g:CFK_{\mathcal{R}}(S^3,K)\rightarrow \mathcal{R}$.
\end{lem}
\begin{proof}
Consider the free-stabilization map 
\[
S^+_{z_{free}}:CFK_{UV}(S^3,K)\rightarrow CFK_{UV}(S^3,K,z_{free}).
\]
Then we have 
\[
f\circ S^+_{z_{free}}\circ \iota_K \sim f\circ \iota_{K,z_{free}}\circ S^+_{z_{free}} \sim f \circ S^+_{z_{free}}.
\]
Since the codomain $f\circ S^+_{z_{free}}$ is $\mathcal{R}$, it induces a chain map 
\[
f_{\mathcal{R}}:CFK_{\mathcal{R}}(S^3,K)=CFK_{UV}(S^3,K)\otimes \mathcal{R} \xrightarrow{(f\circ S^+_{z_{free}}) \otimes \mathbf{id}_{\mathcal{R}}} CFK_{UV}(S^3,\mathbf{unknot},z_{free})\otimes \mathcal{R}.
\]
Since $U^{-1}f_{\mathcal{R}}\vert_{V=0}=U^{-1}(f\circ S^+_{z_{free}})\vert_{V=0}$ and $V^{-1}f_{\mathcal{R}}\vert_{U=0}=V^{-1}(f\circ S^+_{z_{free}})\vert_{U=0}$, and the maps $f$ and $S^+_{z_{free}}$ are local, we deduce that $f_{\mathcal{R}}$ is also local.

Recall that the differential on $CFK_{UV}(S^3,\mathbf{unknot},z_{free})=(\mathbb{F}_2[U,V]\mathbf{x}_+\oplus \mathbb{F}_2[U,V]\mathbf{x}_-,\partial)$ is given by 
\[
\partial \mathbf{x}_+ =0,\,\partial \mathbf{x}_-=UV\mathbf{x}_+.
\]
Since $UV=0$ in $\mathcal{R}$, we can define a projection map 
\[
p:CFK_{UV}(S^3,\mathbf{unknot},z_{free})\otimes \mathcal{R}\rightarrow \mathcal{R}
\]
by $p(\mathbf{x}_+)=1$ and $p(\mathbf{x}_-)=0$. Then the composed map $g=p \circ f_{\mathcal{R}}$ satisfies 
\[
g\circ \iota_K = p \circ f_{\mathcal{R}} \circ \iota_K \sim p \circ f_{\mathcal{R}} = g.
\]
Furthermore, $g$ is a local map due to grading reasons. Therefore $g$ is the desired map.
\end{proof}

\section{Involutive knot Floer homology and involutive bordered Floer homology}
\label{sec:fromsurgerytohfk}

Recall that, for any two bordered 3-manifold $M,N$ with the same boundary, we have a pairing formula 
\[
\widehat{CF}(M\cup N)\simeq \mathbf{Mor}(\widehat{CFD}(-M),\widehat{CFD}(N)).
\]
Note that the cycles in the morphism space correspond to type-D morphisms, and boundaries correspond to nullhomotopic morphisms. Consider the case when $M$ is the 0-framed complement of a knot $K$ and $N$ is the 0-framed solid torus. Then we have $S^3 _{0} (-K)\simeq -M\cup N$, so the pairing formula induces a homotopy equivalence 
\[
\widehat{CF}(S^3 _{0}(-K)) \simeq \mathbf{Mor}(\widehat{CFD}(S^3\backslash K),\widehat{CFD}(T_0)),
\]
where $T_0$ denotes the 0-framed solid torus. Now, by \Cref{lem:Xpairing}, we get a chain map:
\[
\begin{split}
    F_{\X}:\widehat{CF}(S^3 _{0}(-K)) &\xrightarrow{\simeq} \mathbf{Mor}(\widehat{CFD}(S^3\backslash K),\widehat{CFD}(T_0)) \\
    &\xrightarrow{f\mapsto \mathbf{id}\boxtimes f}\mathbf{Hom}(CFA_{UV}(\X_0) \boxtimes \widehat{CFD}(S^3 \backslash K), CFA_{UV}(\X_0) \boxtimes \widehat{CFD}(T_0))  \\
    &\xrightarrow{\text{pairing}} \mathbf{Hom}(CFK_{UV}(S^3,K,z_{free}),CFK_{UV}(S^3,\mathbf{unknot},z_{free})).
\end{split}
\]
On the other hand, by pairing with $\widehat{CFA}(\infty\text{-framed solid torus})$ instead of $CFA_{UV}(\X_0)$, we also get a chain map $$F:\widehat{CF}(S^3_{0}(-K))\rightarrow \mathbf{Hom}(\widehat{CF}(S^3),\widehat{CF}(S^3))=\mathbb{F}_2.$$

\begin{lem}
\label{lem:cobordismlem}
Let $X_{0}(-K)$ be the punctured 0-trace of the knot $-K$, i.e. the 4-manifold obtained by attaching a 0-framed 2-handle to $S^3 \times I$ along $-K\times \{ 1\}$. Then the map $x\mapsto F(x)(1):\widehat{CF}(S^3 _{0}(-K))\rightarrow \widehat{CF}(S^3)$ is the hat-flavored cobordism map induced by the cobordism $X_{0}(-K)$, flipped upside-down.
\end{lem}
\begin{proof}
Discussions in \cite[Section 1.5]{lipshitz2011heegaard} tells us that the map $F:\widehat{CF}(S^3_{0}(-K))\otimes \widehat{CF}(S^3)\rightarrow \widehat{CF}(S^3)$ is the cobordism map induced by the 4-manifold $W_0$ given by
\[
W_0=(\triangle\times T)\cup _{e_1\times T} (e_1\times (S^3 \backslash K))\cup_{e_2\times T} (e_2\times T_\infty) \cup_{e_3 \times T} (e_3 \times T_0),
\]
where $\triangle$ denotes a triangle with edges $e_1,e_2,e_3$, and $T$ denotes a torus. Note that $W_0$ has three boundary components given by $S^3 _{0}(-K)=-(S^3\backslash K)\cup T_0$, $S^3=T_0 \cup T_\infty$, and $S^3=(S^3 \backslash K)\cup T_\infty$. Hence the cobordism map induced by 4-manifold $W$ obtained by gluing a 4-ball to the second boundary, i.e. 
\[
W=W_0 \cup_{T_0\cup T_\infty} B^4,
\]
is the given map $x\mapsto F(x)(1)$. Since $W$ is diffeomorphic to $X_0(K)$, flipped upside-down, the lemma follows.
\end{proof}

The following example explains \Cref{lem:cobordismlem} in the case when $K$ is the unknot.
\begin{exmp}
\label{exmp:toymodel}
Let $K$ be the unknot. Then $S^3 \backslash K\simeq T_0$, $S^3 _0 (-K) \simeq S^1 \times S^2$, and $X_0 (-K)\simeq D^2 \times S^2$. The type-D module of the 0-framed solid torus $T_0$ is freely generated over the torus algebra $\mathcal{A}(T^2)$, which is generated (over $\mathbb{F}_2$) by the set 
\[
\{ \iota_0,\iota_1,\rho_1,\rho_2,\rho_3,\rho_{12},\rho_{23},\rho_{123} \},
\]
by a single element $x$, and the differential is given by $\partial x=\rho_{12}x$. The identity morphism
\[
\mathbf{id}:\widehat{CFD}(S^3 \backslash K)=\widehat{CFD}(T_0)\rightarrow \widehat{CFD}(T_0)
\]
corresponds to the $\frac{1}{2}$-graded generator in $\widehat{CFD}(S^1 \times S^2)\simeq \mathbb{F}_2 \left[ \frac{1}{2} \right] \oplus \mathbb{F}_2 \left[ -\frac{1}{2} \right]$. Note that the $-\frac{1}{2}$-graded generator corresponds to the map $x\mapsto \rho_{12}x$. The cobordism map 
\[
x\mapsto F(x)(1):\widehat{CF}(S^1 \times S^2)\rightarrow \widehat{CF}(S^3)
\]
induced by $D^2 \times S^2$ which bounds $S^1 \times S^2$ is a map of degree $-\frac{1}{2}$, which maps the $\frac{1}{2}$-graded generator (which corresponds to the identity morphism) to $1$ and the $-\frac{1}{2}$-graded generator to 0.
\end{exmp}

\begin{lem}
\label{lem:localmaplem}
Let $K$ be a knot such that $(CFK_{\mathcal{R}}(S^3,K),\iota_K)$ is locally equivalent to the trivial complex. Then there exists a cycle $x\in \widehat{HF}(S^3 _0 (-K))$ of absolute $\mathbb{Q}$-grading $\frac{1}{2}$, which is mapped to the unique homotopy autoequivalence $[\mathbf{id}]\in H_{\ast}(\mathbf{Hom}(\widehat{CF}(S^3),\widehat{CF}(S^3)))$ under the map $F$.
\end{lem}
\begin{proof}
By \Cref{lem:cobordismlem}, we know that the map $x\mapsto F(x)(1):\widehat{CF}(S^3 _{0}(-K))\rightarrow \widehat{CF}(S^3)$ is the hat-flavored cobordism map induced by the cobordism $W$ by flipping the 0-framed 2-handle attaching map along $-K$ upside-down. Recall from the involutive mapping cone formula \cite[Section 22.9]{hendricks2020surgery} that the Heegaard Floer homology of $S^3 _0 (-K)$ is homotopy equivalent to a complex of the form
\[
\widehat{CF}(S^3 _0 (-K),[0])\simeq \mathbf{Cone}(D_0 :\hat{A}_0 \rightarrow \hat{B}_0),
\]
and the involution $\iota_{S^3 _0 (-K)}$ takes the form $\iota_A + \iota_B + H$, where $\iota_A$ and $\iota_B$ are the involutions on $\hat{A}_0$ and $\hat{B}_0$, respectively, induced by $\iota_{-K}$, and $H$ is a certain homotopy between $\iota_B \circ D_0$ and $D_0 \circ \iota_A$. Also, it is shown in \cite[Theorem 15.1]{hendricks2022naturality} that the cobordism map $\widehat{CF}(S^3 _0 (-K),[0])\rightarrow \widehat{CF}(S^3)$ is given by the projection onto $\hat{A}_0$, composed with the inclusion map of $\hat{A}_0$ into $\hat{B}_0$.

Let $g:CFK_{\mathcal{R}}(S^3,\mathbf{unknot}) \rightarrow CFK_{\mathcal{R}}(S^3,K)$ be a local map such that $\iota_K \circ g\sim g\circ \iota_{\mathbf{unknot}}$. Following the proof of \cite[Proposition 3.15(3)]{hendricks2020surgery} shows that choosing a homotopy between $\iota_K \circ g$ and $g \circ \iota_{\mathbf{unknot}}$ induces a local map $F_g:\widehat{CF}(S^1 \times S^2) \rightarrow \widehat{CF}(S^3 _0 (-K),[0])$ satisfying $\iota_{S^3 _0 (-K)} \circ F_g \sim F_g \circ \iota_{S^1 \times S^2}$. Denote by $x_0$ the unique generator of the $\frac{1}{2}$-graded piece of $\widehat{HF}(S^1 \times S^2)$. Since projection to $A_0$ clearly homotopy-commutes with $F_g$, we see from \Cref{exmp:toymodel} that $F_g(x_0)$ is a $\iota_{S^3 _0 (-K)}$-invariant element of $\widehat{HF}(S^3 _0 (-K))$ which is mapped to the generator of $\widehat{HF}(S^3)$ under the cobordism map induced by $W$, proving the lemma.
\end{proof}

Now we can prove \Cref{thm:mainthm1}.
\begin{proof}[Proof of \Cref{thm:mainthm1}]
Given two knots $K_1$ and $K_2$, suppose that $(CFK_{\mathcal{R}}(S^3,K_1 \sharp -K_2),\iota_{K_2 \sharp -K_1})$ is $\iota_K$-locally equivalent to the trivial complex. By \Cref{lem:localmaplem}, there exists a cycle $x\in \widehat{HF}(S^3 _0 (K_2 \sharp -K_1))$ of absolute $\mathbb{Q}$-grading $\frac{1}{2}$, which is invariant under the action of $\iota_{S^3 _0 (K_2 \sharp -K_1)}$ and mapped to the unique homotopy autoequivalence $[\mathbf{id}]\in H_{\ast}(\mathbf{Hom}(\widehat{CF}(S^3),\widehat{CF}(S^3)))$ under the map $F$.

Since we have 
\[
-(S^3 \backslash K_1) \cup (S^3 \backslash K_2) \simeq S^3 _0 (K_2 \sharp -K_1),
\]
we have a pairing theorem 
\[
\widehat{CF}(S^3 _0 (K_2 \sharp -K_1))\simeq \mathbf{Mor}(\widehat{CFD}(S^3 \backslash K_1),\widehat{CFD}(S^3 \backslash K_2)).
\]
Denote by $F_x:\widehat{CFD}(S^3 \backslash K_1)\rightarrow \widehat{CFD}(S^3 \backslash K_2)$ the type-D morphism which corresponds to $x$. Then we have the following homotopy-commutative diagram for any choice of $\iota_{S^3 \backslash K_1}\in \mathbf{Inv}_D(S^3 \backslash K_1)$ and $\iota_{S^3 \backslash K_2}\in \mathbf{Inv}_D(S^3 \backslash K_2)$:
\[
\xymatrix{
\widehat{CFDA}(\mathbf{AZ})\boxtimes \widehat{CFD}(S^3 \backslash K_1) \ar[d]^{\mathbf{id}\boxtimes g}  \ar[rr]^{\iota_{S^3 \backslash K_1}} & &\widehat{CFD}(S^3\backslash K_1)\ar[d]^{g} \\
\widehat{CFDA}(\mathbf{AZ}) \boxtimes \widehat{CFD}(S^3 \backslash K_2) \ar[rr]^{\iota_{S^3 \backslash K_2}} & & \widehat{CFD}(S^3 \backslash K_2)
}
\]
Furthermore, since $F(x)$ corresponds to the identity morphism of $\widehat{CF}(S^3)$, we see that the induced map
\[
\begin{split}
    \widehat{CF}(S^3) &\simeq \widehat{CFA}(T_\infty)\boxtimes \widehat{CFD}(S^3 \backslash K_1) \\
    &\xrightarrow{\mathbf{id}\boxtimes g} \widehat{CFA}(T_\infty)\boxtimes \widehat{CFD}(S^3 \backslash K_2) \simeq \widehat{CF}(S^3)
\end{split}
\]
is homotopic the identity morphism. 

Now suppose that we have a type-D morphism $g:\widehat{CFD}(S^3 \backslash K_1)\rightarrow \widehat{CFD}(S^3 \backslash K_2)$ which satisfies the conditions of \Cref{thm:mainthm1} for some choices of $\iota_{S^3 \backslash K_1}\in \mathbf{Inv}_D(S^3 \backslash K_1)$ and $\iota_{S^3 \backslash K_2}\in \mathbf{Inv}_D(S^3 \backslash K_2)$. By taking a box tensor product with an involution $\iota_{T_\infty \backslash P}\in \mathbf{Inv}(T_\infty \backslash P)$ of the type DA bimodule $\widehat{CFDA}(T_\infty \backslash P)$ of the exterior of the connected-sum pattern $P$ induced by $-K_1$, we may replace $K_1$ with $K_1 \sharp -K_1$ and $K_2$ with $K_2 \sharp -K_1$ without any loss of generality (see the discussion below the proof for details). Then, after pairing with $\widehat{CFDA}(\X_0)$, we get the following homotopy-commutative diagram.
\[
\xymatrix{
CFK_{UV}(S^3,K_1 \sharp -K_1,z_{free})\ar[rr]^{F_{\X}(x)} \ar[d]^{\simeq} & & CFK_{UV}(S^3,K_2 \sharp -K_1,z_{free}) \ar[d]^{\simeq} \\
CFA_{UV}(\X_0) \boxtimes \widehat{CFD}(S^3\backslash (K_1 \sharp -K_1)) \ar[rr]^{\mathbf{id}\boxtimes F_x} \ar[d]^{\mathbf{conj}_{\X}\boxtimes \mathbf{id}} & & CFA_{UV}(\X_0) \boxtimes \widehat{CFD}(S^3 \backslash  (K_2 \sharp -K_1)) \ar[d]^{\mathbf{conj}_{\X} \boxtimes \mathbf{id}} \\
CFA_{UV}(\bar{\X}_0) \boxtimes \widehat{CFD}(S^3\backslash (K_1 \sharp -K_1)) \ar[rr]^{\mathbf{id}\boxtimes F_x} \ar[d]^{\phi_{\ast}\boxtimes \mathbf{id}} & & CFA_{UV}(\bar{\X}_0) \boxtimes \widehat{CFD}(S^3 \backslash  (K_2 \sharp -K_1)) \ar[d]^{\phi_{\ast} \boxtimes \mathbf{id}} \\
CFA_{UV}(\phi(\bar{\X}_0)) \boxtimes \widehat{CFD}(S^3\backslash (K_1 \sharp -K_1)) \ar[rr]^{\mathbf{id}\boxtimes F_x} \ar[d]^{\simeq} & & CFA_{UV}(\phi(\bar{\X}_0)) \boxtimes \widehat{CFD}(S^3 \backslash  (K_2 \sharp -K_1)) \ar[d]^{\simeq} \\
{\begin{array}{@{}c@{}}CFA_{UV}(\phi(\bar{\X}_0)) \boxtimes \widehat{CFDA}(\overline{\mathbf{AZ}}) \\ \boxtimes \widehat{CFDA}( \mathbf{AZ}) \boxtimes \widehat{CFD}(S^3\backslash (K_1 \sharp -K_1))\end{array}} \ar[rr]^{\mathbf{id}\boxtimes \mathbf{id}\boxtimes \mathbf{id} \boxtimes F_x} \ar[d]^{\iota_{\X} \boxtimes \iota_{S^3 \backslash (K_1 \sharp -K_1)}}& & {\begin{array}{@{}c@{}}CFA_{UV}(\phi(\bar{\X}_0)) \boxtimes \widehat{CFDA}(\overline{\mathbf{AZ}}) \\ \boxtimes \widehat{CFDA}(\mathbf{AZ}) \boxtimes \widehat{CFD}(S^3 \backslash  (K_2 \sharp -K_1))\ar[d]^{\iota_{\X} \boxtimes \iota_{S^3 \backslash  (K_2 \sharp -K_1)}}\end{array}}\\
CFA_{UV}(\X_0) \boxtimes \widehat{CFD}(S^3\backslash (K_1 \sharp -K_1)) \ar[rr]^{\mathbf{id}\boxtimes F_x} \ar[d]^{\simeq} & & CFA_{UV}(\X_0) \boxtimes \widehat{CFD}(S^3 \backslash  (K_2 \sharp -K_1)) \ar[d]^{\simeq} \\
CFK_{UV}(S^3,K_1 \sharp -K_1,z_{free}) \ar[rr]^{F_{\X}(x)}  & & CFK_{UV}(S^3,K_2 \sharp -K_1,z_{free})
}
\]
By \Cref{lem:freestabinv}, the compositions of vertical maps on the two columns of the above diagram are $\iota_{K_1 \sharp -K_1,z_{free}}$ and $\iota_{K_2 \sharp -K_1,z_{free}}$, respectively, which implies that $F_{\X}(x)\circ \iota_{K_1 \sharp -K_1,z_{free}} \sim \iota_{K_2 \sharp -K_1,z_{free}} \circ F_{\X}(x)$. Since $K_1 \sharp -K_1$ is slice, we should have a local map 
\[
G:CFK_{UV}(S^3,\mathbf{unknot},z_{free})\rightarrow CFK_{UV}(S^3,K_1 \sharp -K_1,z_{free})
\]
satisfying $G\sim \iota_{K_1\sharp -K_1,z_{free}} \circ G$. Hence, by \Cref{lem:eliminationlem}, we have an $\iota_K$-local chain map 
\[
f:\mathcal{R}\rightarrow CFK_{\mathcal{R}}(S^3,K).
\]

Now, since our argument can also be applied to $-K$ instead of $K$, we should also have an $\iota_K$-local chain map 
\[
f^{\prime}:CFK_{\mathcal{R}}(S^3,-K)\rightarrow \mathcal{R}.
\]
Therefore $[CFK_{\mathcal{R}}(S^3,K),\iota_K]$ is locally equivalent to the trivial complex $\mathcal{R}$.
\end{proof}

Now suppose that we have two bordered 3-manifolds $M$ and $N$, where $M$ has one torus boundary $\partial M$ and $N$ has two torus boundaries, $\partial _1 N$ and $\partial _2 N$. Choose any $\iota_M \in \mathbf{Inv}_D(M)$ and $\iota_N \in \mathbf{Inv}(N)$, so that we have type-D and type-DA homotopy equivalences 
\[
\begin{split}
    \iota_M &: \widehat{CFDA}(\mathbf{AZ})\boxtimes \widehat{CFD}(M) \rightarrow \widehat{CFD}(M), \\
    \iota_N &: \widehat{CFDA}(\mathbf{AZ})\boxtimes \widehat{CFDA}(N) \boxtimes \widehat{CFDA}(\overline{\mathbf{AZ}}) \rightarrow \widehat{CFDA}(N),
\end{split}
\]
where the boundary components $\partial _1 N$ and $\partial _2 N$ are considered as type-A and type-D boundaries, respectively. Recall that we have a pairing theorem for computing $\widehat{CFD}(M\cup N)$, where we identify $\partial M$ with $\partial_1 N$:
\[
\widehat{CFD}(M\cup N) \simeq \widehat{CFDA}(N)\boxtimes \widehat{CFD}(M).
\]
Then our choice of $\iota_M$ and $\iota_N$ induces a homotopy equivalence $\iota_{\iota_M,\iota_N}$ for $\widehat{CFD}(M\cup N)$ as follows:
\[
\begin{split}
\iota_{\iota_M,\iota_N}: \widehat{CFDA}(\mathbf{AZ})\boxtimes \widehat{CFD}(M\cup N) &\simeq \widehat{CFDA}(\mathbf{AZ})\boxtimes \widehat{CFDA}(N) \boxtimes \widehat{CFD}(M) \\
&{\begin{array}{@{}c@{}}\simeq \widehat{CFDA}(\mathbf{AZ})\boxtimes \widehat{CFDA}(N)\boxtimes \widehat{CFDA}(\overline{\mathbf{AZ}}) \\ \boxtimes\widehat{CFDA}(\mathbf{AZ})\boxtimes \widehat{CFD}(M)\end{array}} \\
&\xrightarrow{\iota_M \boxtimes \iota_N} \widehat{CFDA}(N)\boxtimes \widehat{CFD}(M) \simeq \widehat{CFD}(M\cup N).
\end{split}
\]
Following the proof of \cite[Theorem 5.1]{hendricks2019involutivebordered}, we immediately see that $\iota_{\iota_M,\iota_N}\in \mathbf{Inv}_D(M\cup N)$. Using this fact, we can now prove \Cref{thm:mainthm2}.
\begin{proof}[Proof of \Cref{thm:mainthm2}]
Let $K_1$ and $K_2$ be two knots satisfying the given assumptions. Then, by \Cref{thm:mainthm1}, there exists a type D morphism $$g:\widehat{CFD}(S^3\backslash K_1)\rightarrow \widehat{CFD}(S^3 \backslash K_2)$$ which fits into the following homotopy-commutative diagram for any choice of $\iota_{S^3 \backslash (K_1 \sharp -K_2)}\in \mathbf{Inv}_D(S^3 \backslash (K_1 \sharp -K_2))$.
\[
\xymatrix{
\widehat{CFDA}(\mathbf{AZ})\boxtimes \widehat{CFD}(S^3 \backslash K_1) \ar[d]^{\mathbf{id}\boxtimes g}  \ar[rr]^{\iota_{S^3 \backslash K_1}} & &\widehat{CFD}(S^3\backslash K_1)\ar[d]^{g} \\
\widehat{CFDA}(\mathbf{AZ}) \boxtimes \widehat{CFD}(S^3 \backslash K_2) \ar[rr]^{\iota_{S^3 \backslash K_2}} & & \widehat{CFD}(S^3 \backslash K_2)
}
\]
Furthermore, the induced chain map 
\[
\begin{split}
    \widehat{CF}(S^3) &\simeq \widehat{CFA}(T_\infty)\boxtimes \widehat{CFD}(S^3 \backslash K_1) \\
    &\xrightarrow{\mathbf{id}\boxtimes g} \widehat{CFA}(T_\infty)\boxtimes \widehat{CFD}(S^3 \backslash K_2) \simeq \widehat{CF}(S^3)
\end{split}
\]
is a homotopy equivalence.

Now let $N_1 =T_\infty \backslash P$ be the 0-framed exterior of the given pattern $P$ inside the $\infty$-framed solid torus. Then  the union of $N$ (glued along its 0-framed boundary) with $T_\infty$ is again $T_\infty$.
Hence, if we denote the type D morphism 
\[
\begin{split}
\widehat{CFD}(S^3 \backslash P(K_1)) &\simeq \widehat{CFDA}(N)\boxtimes \widehat{CFD}(S^3 \backslash K_1) \\
&\xrightarrow{\mathbf{id} \boxtimes g} \widehat{CFDA}(N) \boxtimes \widehat{CFD}(S^3 \backslash K_2) \\
&\simeq \widehat{CFD}(S^3 \backslash P(K_2))
\end{split}
\]
by $g_0$, then the induced map 
\[
\begin{split}
\widehat{CF}(S^3) &\simeq \widehat{CFA}(T_\infty)\boxtimes \widehat{CFD}(S^3\backslash P(K_1\sharp K_2\sharp -K_2)) \\
&\xrightarrow{\mathbf{id}\boxtimes g_0} \widehat{CFA}(T_\infty) \boxtimes \widehat{CFD}(S^3\backslash P(K_2)) \\
&\simeq \widehat{CF}(S^3)
\end{split}
\]
is homotopic to identity. Furthermore, we have a following homotopy-commutative diagram.
\[
\xymatrix{
\widehat{CFDA}(\mathbf{AZ})\boxtimes \widehat{CFD}(S^3 \backslash P(K_1))\ar[rr]^{\mathbf{id}\boxtimes g_0}\ar[d]^{\simeq} & &\widehat{CFDA}(\mathbf{AZ})\boxtimes \widehat{CFD}(S^3 \backslash P(K_2))\ar[d]^{\simeq} \\
\widehat{CFDA}(\mathbf{AZ})\boxtimes \widehat{CFDA}(N)\boxtimes \widehat{CFD}(S^3 \backslash K_1)\ar[rr]^{\mathbf{id}\boxtimes \mathbf{id}\boxtimes g}\ar[d]^{\simeq} & &\widehat{CFDA}(\mathbf{AZ})\boxtimes \widehat{CFDA}(N)\boxtimes \widehat{CFD}(S^3 \backslash K_2)\ar[d]^{\simeq} \\
{\begin{array}{@{}c@{}}\widehat{CFDA}(\mathbf{AZ})\boxtimes \widehat{CFDA}(N)\boxtimes \widehat{CFDA}(\overline{\mathbf{AZ}}) \\  \boxtimes\widehat{CFDA}(\mathbf{AZ}) \boxtimes \widehat{CFD}(S^3 \backslash K_1)\end{array}} \ar[rr]^{\mathbf{id}\boxtimes\mathbf{id}\boxtimes\mathbf{id}\boxtimes\mathbf{id}\boxtimes g}\ar[d]^{\iota_N \boxtimes \iota_{S^3\backslash K_1}} & &{\begin{array}{@{}c@{}}\widehat{CFDA}(\mathbf{AZ})\boxtimes \widehat{CFDA}(N)\boxtimes \widehat{CFDA}(\overline{\mathbf{AZ}}) \\ \boxtimes\widehat{CFDA}(\mathbf{AZ})\boxtimes \widehat{CFD}(S^3 \backslash K_2)\end{array}}\ar[d]^{\iota_N \boxtimes \iota_{S^3 \backslash K_2}}\\
\widehat{CFDA}(N)\boxtimes \widehat{CFD}(S^3 \backslash K_1)\ar[rr]^{\mathbf{id}\boxtimes g}\ar[d]^{\simeq} & &\widehat{CFDA}(N)\boxtimes \widehat{CFD}(S^3 \backslash K_2) \ar[d]^{\simeq} \\
\widehat{CFD}(S^3 \backslash P(K_1))\ar[rr]^{g_0} & &\widehat{CFD}(S^3 \backslash P(K_2))
}
\]
The compositions of vertical maps on both sides of the above diagram are $\iota_{\iota_N,\iota_{S^3 \backslash K_1}}$ and $\iota_{\iota_N,\iota_{S^3 \backslash K_2}}$, which are contained in $\mathbf{Inv}_D(S^3 \backslash P(K_1))$ and $\mathbf{Inv}_D(S^3 \backslash P(K_2))$, respectively. Also, since our assumption is symmetric on the choices of $K_1$ and $K_2$, we can repeat our argument with $K_1$ and $K_2$ swapped. Hence, by \Cref{thm:mainthm1}, we deduce that $(CFK_{\mathcal{R}}(S^3,P(K_1)\sharp -P(K_2)),\iota_{P(K_1) \sharp -P(K_2)})$ is $\iota_K$-locally equivalent to the trivial complex.
\end{proof}

\section{An explicit formula for the hat-flavored truncation of $\iota_K$}
\label{sec:formula}

Recall that we had the bordered Heegaard diagram $\X$; write $\X=(\Sigma,\boldsymbol\alpha,\boldsymbol\beta,\{z,z_{free}\},w)$. We can add one more free basepoint $w_{free}$ to the component of $\Sigma\backslash \left( \cup_{c\in \alpha\cup\beta} c\right)$ containing $z_{free}$ to get a new diagram $\Y=(\Sigma,\boldsymbol\alpha,\boldsymbol\beta,\{z,z_{free}\},\{w,w_{free}\})$. As we modified $\X$ by Heegaard moves to get a nice diagram $\X_0$, we can do the same process to $\Y$ to get a nice diagram $\Y$. By counting holomorphic disks on $\Y_0$ which does not algebraically intersect $z_{free}$ and $w_{free}$, and recording their algebraic intersection numbers with $z$ and $w$ by formal variables $U$ and $V$, respectively, we can get a well-defined type-A module $CFA_{UV}(\Y_0)$. Note that, by construction, we have 
\[
CFA_{UV}(\X_0) \simeq CFA_{UV}(\Y_0).
\]

Recall that the proof of the pairing theorem 
\[
\widehat{CF}(-M_1 \cup M_2 ) \simeq \mathbf{Mor}(\widehat{CFD}(M_1),\widehat{CFD}(M_2))
\]
relies on the observation that $-M_1 \cup \overline{\mathbf{AZ}} \simeq -M_1$. Denote by $\Y_\infty$ the 4-pointed nice bordered diagram obtained by gluing $\Y_0$ with a cylinder whose boundaries have framing $0$ and $\infty$. Since $\Y_\infty$ should also satisfy $\widehat{CFD}(-\Y_\infty \cup \overline{\mathbf{AZ}}) \simeq \widehat{CFD}(-\Y_\infty)$ and the type-D and type-A modules associated to $\Y_0$ and $\X_0$ are homotopy equivalent, we see that 
\[
\widehat{CFL}(S^2 \times S^1,L_2) \simeq \widehat{CFA}(-\Y_\infty) \boxtimes \widehat{CFD}(S^3 \backslash K) \simeq  \mathbf{Mor}(\widehat{CFD}(\Y_\infty),\widehat{CFD}(T_\infty,\nu)),
\]
where $\nu$ denotes the longitudinal knot inside the $\infty$-framed solid torus $T_\infty$ and $L_2$ denotes the 2-component link ${p,q}\times S^1$ in $S^2 \times S^1$ for two points $p,q\in S^2$. 

Here, $L_2$ is endowed with an orientation so that its total homology class $[L_2]\in H_1 (S^2 \times S^1;\mathbb{Z})$ vanishes. Hence $L_2$ is nullhomologous, which tells us that its link Floer homology (at the unique spin structure of $S^2 \times S^1$ has well-defined $\mathbb{Z}$-valued Maslov and (collapsed) Alexander gradings. These gradings should be compatible with the natural gradings of $\widehat{CFD}(\Y_\infty)$ and $\widehat{CFD}(T_\infty,\nu)$; note that the grading on $\widehat{CFD}(\Y_\infty)$ can be defined as in \Cref{eqn:grading}.

\begin{lem}
\label{lem:fourgen}
$\widehat{HFL}(S^2 \times S^1,L_2)$ is generated by four elements, which are supported on the unique spin structure of $S^2 \times S^1$ and lie on bidegrees $(0,0),(0,0),(1,1),(-1,-1)$, respectively. 
\end{lem}
\begin{proof}
Write $L_2 = A\cup B$ and choose $\mathbf{z}$-basepoint $z_1,z_2$ and $\mathbf{w}$-basepoints $w_1,w_2$ on $L_2$ so that $z_1,w_1\in A$ and $z_2,w_2\in B$. We will compute the link Floer homology $CFL_{UV}(S^2 \times S^1,L_2)$ of the basepointed link $(L_2,\{z_1,z_2\},\{w_1,w_2\})$, where the differential records the algebraic intersections of holomorphic disks with the basepoints $z_1,w_1,z_2,w_2$ by $U,V,0,0$, respectively. Note that truncating it by $U=V=0$ and taking homology gives $\widehat{HFL}(S^2 \times S^1,L_2)$.

Consider the Heegaard diagram in \Cref{fig:heegaarddiagram}. Since we are counting disks which does not intersect $z_2$ and $w_2$ algebraically, the given diagram is nice, so all relevant holomorphic disks are represented by either bigons or squares which do not contain $z_2$ and $w_2$. Thus we see that $CFL_{UV}(S^2 \times S^1,L_2)$ is generated by the intersection points $xc,xd,yc,yd$, and the differential is given by 
\[
\begin{split}
    \partial(xd) &= Uxc+Vyd, \\
    \partial(xc) &= V(xd+yc), \\
    \partial(yd) &= U(xd+yc), \\
    \partial(xd+yc) &= 0.
\end{split}
\]
Since $U$ and $V$ act on the bigrading by $(-2,-1)$ and $(0,1)$, and the differential $\partial$ lowers the Maslov grading by $1$ and leaves the collapsed Alexander grading invariant, we see that $xd$ and $xd+yc$ have bidegree $(0,0)$, $xc$ has bidegree $(1,1)$, and $yd$ has bidegree $(-1,-1)$. Therefore, after truncating by $U=V=0$, we get four generators $xd,xc,yd,xd+yc$ of $\widehat{HFL}(S^2 \times S^1,L_2)$, which lie on bidegrees $(0,0),(0,0),(1,1),(-1,-1)$, respectively, as desired.
\end{proof}

\begin{figure}[hbt]
    \centering
    \includegraphics[width=0.4\textwidth]{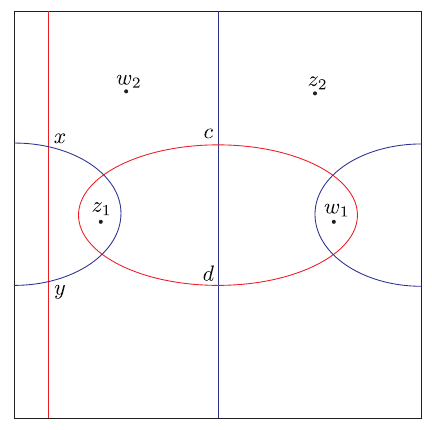}
    \caption{A 4-pointed Heegaard diagram representing the 2-component link $L_2$.}
\label{fig:heegaarddiagram}
\end{figure}

We define a type-D morphism
\[
G_\infty:\widehat{CFD}(\Y_\infty)\rightarrow \widehat{CFD}(T_\infty,\nu)
\]
as follows. We start with a Heegaard diagram $\Y_0$. If we denote by $H_0=(\Sigma,\boldsymbol\alpha,\boldsymbol\beta,z_{free},w)$ the doubly-pointed Heegaard diagram for the pair $(T_\infty,\nu)$ and the diagram we get by quasi-stabilizing it as $H^{qst}_0$, then we have a 2-handle map 
\[
\widehat{CFD}(\Y_0)\rightarrow \widehat{CFD}(H^{qst}_0).
\]
Furthermore, the proof of \cite[Proposition 5.3]{zemke2017quasistabilization} tells us that we can define the ``quasi-destabilization map''
\[
\widehat{CFD}(H^{qst}_0)\rightarrow \widehat{CFD}(T_\infty,\nu).
\]
We define $G_\infty$ as the composition of the above two maps.

\begin{figure}[hbt]
    \centering
    \includegraphics[width=0.8\textwidth]{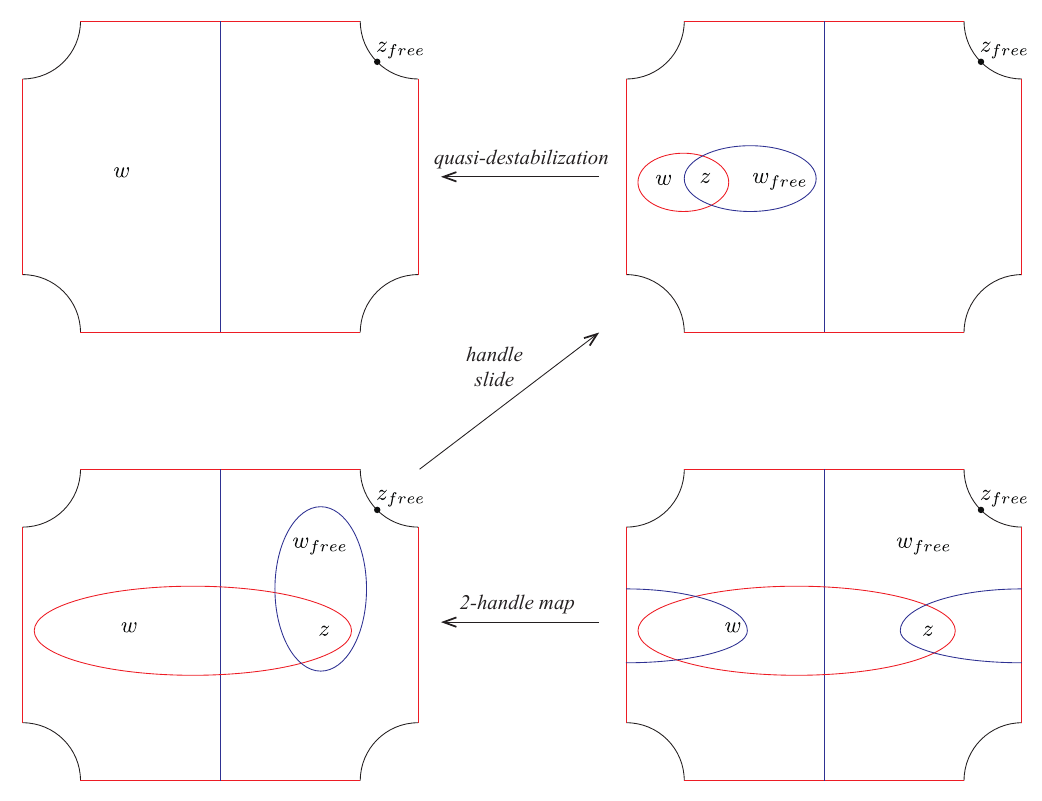}
    \caption{Upper left, the diagram $H_0$. Upper right, the diagram $H^{qst}_0$. Lower left, a result of performing a handleslide to $H^{qst}_0$. Lower right, the diagram $\Y$.}
\label{fig:mapQ}
\end{figure}

Then, for any knot $K$, the induced map 
\[
\begin{split}
    \widehat{CFL}(S^3,K\cup \mathbf{unknot}) &\simeq \widehat{CFA}(S^3 \backslash K) \boxtimes \widehat{CFD}(\Y_\infty) \\
    &\xrightarrow{\mathbf{id}\boxtimes G_\infty} \widehat{CFA}(S^3 \backslash K) \boxtimes \widehat{CFD}(T_\infty,\nu) \simeq \widehat{CFK}(S^3,K)
\end{split}
\]
is homotopic to the cobordism map $F_K$ induced by the trivial saddle cobordism from $K\cup \mathbf{unknot}$ to $K$, as drawn in \Cref{fig:pairofpantscob}. Furthermore, we can also define type-D endomorphisms 
\[
\Phi^D_{\Y}, \Psi^D_{\Y} :\widehat{CFD}(\Y_\infty)\rightarrow \widehat{CFD}(\Y_\infty)
\]
using quasi-destabilization maps and (similarly defined) quasi-stabilization maps, as follows. Given a bordered diagram $H_Y=(\Sigma,\boldsymbol\alpha,\boldsymbol\beta,\{z,z_{free}\},\{w,w_{free}\})$ representing $\Y_\infty$, we ($\alpha$-)quasi-stabilize it near the point $z$ to get a new diagram $H^{qst}_Y$, which introduces a new pair $(z^{\prime},w^{\prime})$ of basepoints, and then we quasi-destabilize it to eliminate the basepoints $z,w$ and rename $z^{\prime},w^{\prime}$ as $z,w$, respectively, to obtain $H_Y$ again. We define the resulting map as $\Psi_Y$, i.e.
\[
\begin{split}
    \Psi_Y:\widehat{CFD}(\Y_\infty) &= \widehat{CFD}(H_Y) \\
    &\xrightarrow{\text{quasi-stabilization}} \widehat{CFD}(H^{qst}_Y) \\
    &\xrightarrow{\text{quasi-destabilization}} \widehat{CFD}(H_Y) = \widehat{CFA}(\Y_0).
\end{split}
\]
We omit the construction of $\Phi^D_{\Y}$, since it is similar to the construction of $\Psi^D_{\Y}$. The definition of $\Phi^D_{\Y}$ and $\Psi^D_{\Y}$ are not natural, i.e. depends on the choices of auxiliary data. However, by the pairing theorem for triangles, we know that the map 
\[
\begin{split}
    \widehat{CFL}(S^3,K\cup\mathbf{unknot}) &\simeq \widehat{CFA}(S^3 \backslash K) \boxtimes \widehat{CFD}(\Y_\infty) \\
    &\xrightarrow{\mathbf{id}\boxtimes \Phi_Y } \widehat{CFA}(S^3 \backslash K) \boxtimes \widehat{CFD}(\Y_\infty) \simeq \widehat{CFL}(S^3,K\cup\mathbf{unknot}).
\end{split}
\]
is homotopic to the basepoint action $\Phi_{K\cup\mathbf{unknot},K}$ corresponding to the basepoint $z$ on the link $K\cup\mathbf{unknot}$, for any knot $K$. A similar statement also holds for $\Psi^D_{\Y}$ as well.

\begin{figure}[hbt]
    \centering
    \includegraphics[width=0.4\textwidth]{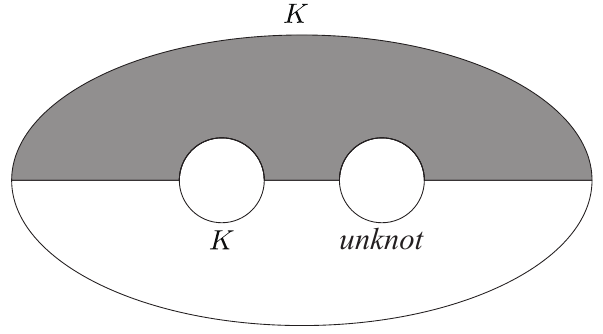}
    \caption{A decoration on the trivial saddle cobordism from $K \cup \mathbf{unknot}$ to $K$. Note that this cobordism can be seen as the composition of a quasi-stabilization followed by a saddle move.}
\label{fig:pairofpantscob}
\end{figure}

\begin{lem}
\label{lem:bypass}
For any knot $K$, we have $F_K \circ \Phi_{K\cup\mathbf{unknot},K}\sim \Phi_K \circ F_K$ and $F_K \circ \Psi_{K\cup\mathbf{unknot},K} \sim \Psi_K \circ F_K$.
\end{lem}
\begin{proof}
Bypass relation \cite[Lemma 1.4]{zemke2019connected}, applied as shown in \Cref{fig:bypass}, gives the equality 
\[
\Psi_K \circ F_K \sim F_K \circ \Psi_{K\cup\mathbf{unknot},K} + F_K \circ \Psi_{K\cup\mathbf{unknot},\mathbf{unknot}},
\]
where $\Psi_{K\cup\mathbf{unknot},\mathbf{unknot}}$ denotes the basepoint action associated to $w_{free} \in \mathbf{unknot}$. Since the basepoint actions for the unknot are trivial, the lemma follows. The same argument also proves the commutation result for $\Phi$ actions.
\end{proof}

\begin{figure}[hbt]
    \centering
    \includegraphics[width=\textwidth]{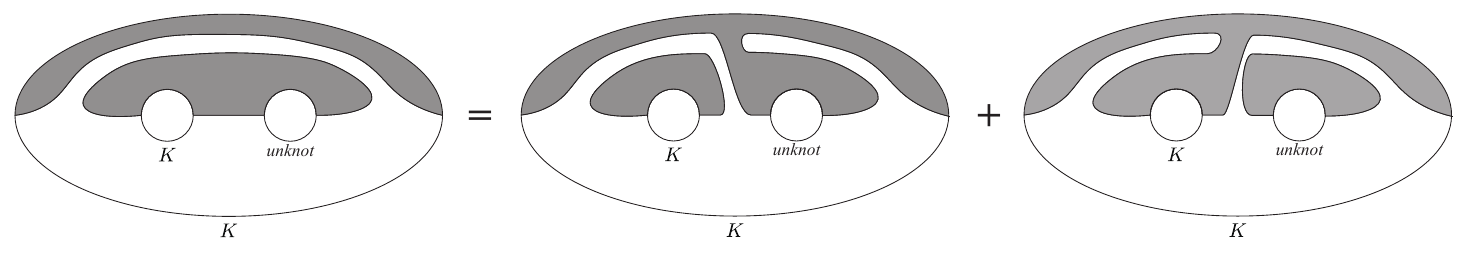}
    \caption{A Bypass relation applied to the saddle cobordism from $K\cup \mathbf{unknot}$ to $K$, with a decoration as shown in \Cref{fig:pairofpantscob}.}
\label{fig:bypass}
\end{figure}

\begin{lem}
\label{lem:basis}
The type-D morphisms $G_\infty$, $G_\infty\circ \Phi^D_{\Y}$, $G_\infty\circ \Psi^D_{\Y}$, and $G_\infty\circ (1+\Phi^D_{\Y}\Psi^D_{\Y})$ form a basis of $H_{\ast}(\mathbf{Mor}(\widehat{CFD}(\Y_\infty),\widehat{CFD}(T_\infty,\nu)))$. Furthermore, they lie on bidegrees $(0,0),(1,1),(-1,-1),(0,0)$, respectively.
\end{lem}
\begin{proof}
By \Cref{lem:fourgen}, we only have to show that the homotopy classes given type-D morphisms are linearly independent, so assume that they are linearly dependent. Then for any knot $K$, the endomorphisms $F_K,F_K\circ \Phi_{K\cup\mathbf{unknot},K},F_K\circ \Psi_{K\cup\mathbf{unknot},K}$, and $F_K\circ (1+\Phi_{K\cup\mathbf{unknot},K}\Psi_{K\cup\mathbf{unknot},K})$ should be linearly dependent up to homotopy. By \Cref{lem:bypass} and the fact that $F_K$ has a homotopy right inverse (which follows from the fact that the trivial saddle cobordism from $K\cup\mathbf{unknot}$ to $K$ has a right inverse) would imply that the endomorphisms 
\[
\mathbf{id},\Phi_K,\Psi_K,1+\Phi_K\Psi_K
\]
of $\widehat{CFK}(S^3,K)$ should also be linearly dependent up to homotopy.

Now consider the case when $K$ is the figure-eight knot. Then $\widehat{CFK}(S^3,K)$ is generated by five elements, say $a,b,c,d,x$. The basepoint actions are given by $\Phi_K(a)=b$, $\Phi_K(c)=d$, $\Psi_K(a)=c$, $\Psi_K(b)=d$, and all other generators are mapped to zero. Thus we see that the endomorphisms $\mathbf{id},\Phi_K,\Psi_K,1+\Phi_K\Psi_K$ are linearly independent up to homotopy, a contradiction.
\end{proof}

Recall that mimicking the construction of $\iota_\X$ gives a bordered involution 
\[
\iota_\Y:\widehat{CFDA}(\mathbf{AZ})\boxtimes \widehat{CFD}(\Y_\infty)\rightarrow \widehat{CFD}(\Y_\infty),
\]
which is a homotopy equivalence which satisfies the property that the induced map 
\[
\begin{split}
    \widehat{CFL}(S^3,K\cup\mathbf{unknot}) &\simeq \widehat{CFA}(S^3 \backslash K)\boxtimes \widehat{CFD}(\Y_\infty) \\
    &\simeq \widehat{CFA}(S^3 \backslash K)\boxtimes \widehat{CFDA}(\overline{\mathbf{AZ}})\boxtimes \widehat{CFD}(\mathbf{AZ}) \boxtimes \widehat{CFD}(\Y_\infty) \\
    &\xrightarrow{\iota_{S^3 \backslash K}\boxtimes \iota_{\Y}} \widehat{CFA}(S^3 \backslash K)\boxtimes \widehat{CFD}(\Y_\infty) \simeq \widehat{CFL}(S^3,K\cup\mathbf{unknot})
\end{split}
\]
is homotopic to the involution $\iota_{K\cup\mathbf{unknot}}$ of the link Floer homology of $K\cup\mathbf{unknot}$. 

On the other hand, the type-D module $\widehat{CFD}(T_\infty,\nu)$ is generated by a single element, say $x$, and the differential is trivial. This implies that $\widehat{CFDA}(\mathbf{AZ})\boxtimes \widehat{CFD}(T_\infty,\nu)$ is not homotopy equivalent to $\widehat{CFD}(T_\infty,\nu)$. In fact, $\widehat{CFDA}(\mathbf{AZ})\boxtimes \widehat{CFD}(T_\infty,\nu)$ is homotopy equivalent to a type-D module generated by five elements, say $a,b,c,d,e$, where the differential is given by 
\begin{equation}
\label{eqn:diff}
\partial a=0,\,\partial b=\rho_1 a+\rho_3 c,\,\partial c=\rho_2 d,\,\partial d=\rho_1 e,\,\partial e=0.    
\end{equation}
Since $a$ is a cycle, the map 
\[
f:\widehat{CFD}(T_\infty,\nu)\rightarrow \widehat{CFDA}(\mathbf{AZ})\boxtimes \widehat{CFDA}(T_\infty,\nu)
\]
defined by $f(x)=a$ commutes with the differential on both sides, and thus is a well-defined type-D morphism.

\begin{lem}
\label{lem:morphismf}
The type-D morphism $f\circ G_\infty \circ \iota_{\Y}$ is homotopic to either $\mathbf{id}\boxtimes G_\infty$ or $\mathbf{id}\boxtimes (G_\infty \circ (1+\Phi^D_{\Y} \Psi^D_{\Y}))$.
\end{lem}
\begin{proof}
For simplicity, write $g=f\circ G_\infty \circ \iota_{\Y}$. Then for any knot $K$, we have an induced map 
\[
\begin{split}
    \widehat{CFL}(S^3,K\cup\mathbf{unknot}) &\simeq \widehat{CFA}(S^3 \backslash K)\boxtimes \widehat{CFDA}(\overline{\mathbf{AZ}})\boxtimes \widehat{CFDA}(\mathbf{AZ}) \boxtimes \widehat{CFD}(\Y_\infty) \\
    &\xrightarrow{\mathbf{id}\boxtimes \mathbf{id}\boxtimes g} \widehat{CFA}(S^3 \backslash K)\boxtimes \widehat{CFDA}(\overline{\mathbf{AZ}})\boxtimes \widehat{CFDA}(\mathbf{AZ}) \boxtimes \widehat{CFD}(T_\infty) \\
    &\simeq \widehat{CFL}(S^3,K),
\end{split}
\]
which we will denote as $g_K$. Then, by construction, we have 
\[
g_K \sim \tilde{f}\circ F_K \circ \iota_{K\cup\mathbf{unknot}},
\]
where $\tilde{f}$ is the map defined as 
\[
\begin{split}
    \widehat{CFK}(S^3,K) &\simeq \widehat{CFA}(S^3 \backslash K)\boxtimes \widehat{CFD}(T_\infty,\nu) \\
    &\xrightarrow{\iota^{-1}_{S^3\backslash K}\boxtimes f} \widehat{CFA}(S^3 \backslash K)\boxtimes \widehat{CFDA}(\overline{\mathbf{AZ}})\boxtimes \widehat{CFDA}(\mathbf{AZ}) \boxtimes\widehat{CFD}(T_\infty,\nu) \\
    &\xrightarrow{\mathbf{id}\boxtimes \Omega\boxtimes \mathbf{id}} \widehat{CFA}(S^3 \backslash K)\boxtimes \widehat{CFD}(T_\infty,\nu)\simeq \widehat{CFK}(S^3,K),
\end{split}
\]
where $\Omega$ is the homotopy equivalence 
\[
\widehat{CFDA}(\overline{\mathbf{AZ}})\boxtimes \widehat{CFDA}(\mathbf{AZ})\simeq \widehat{CFDA}(\mathbb{I}),
\]
which is unique up to homotopy due to homotopy rigidity \cite[Lemma 4.4]{hendricks2019involutivebordered}. It is easy to check, using a bypass relation, that $F_K \circ \iota_{K\cup\mathbf{unknot}} \sim \iota_K \circ F_K$. Hence we get 
\[
g_K \sim \tilde{f} \circ \iota_K \circ F_K.
\]

We now consider the case when $K$ is the unknot. Then the 0-framed knot complement $S^3 \backslash K$ is the 0-framed solid torus $T_0$. Recall that $\widehat{CFDA}(\mathbf{AZ}) \boxtimes\widehat{CFD}(T_\infty,\nu)$ is homotopic to the type-D module $M_D$ generated by $a,b,c,d,e$, where the differential is given as in \Cref{eqn:diff}, and the image of the generator $x$ of $\widehat{CFD}(T_\infty,\nu)$ is $a$. This means that there exists a type-D homotopy equivalence 
\[
h_D:\widehat{CFDA}(\mathbf{AZ})\boxtimes \widehat{CFD}(T_\infty,\nu) \rightarrow M_D
\]
such that $(h_D \circ f)(x)=a$. On the other hand, the type-A module $\widehat{CFA}(S^3 \backslash K)$, which is homotopy equivalent to $\widehat{CFDA}(\overline{\mathbf{AZ}})\boxtimes \widehat{CFDA}(\mathbf{AZ})$ via $\iota_{S^3 \backslash K}$, is generated by one element, say $y$, and the $A_\infty$ operations are given by 
\[
m_{3+i} (y,\rho_2,\overbrace{\rho_{12},\cdots,\rho_{12}}^{i},\rho_1)=y \text { for each } i\ge 0.
\]
Hence the chain map 
\[
\begin{split}
    m:\widehat{CFK}(S^3,K) &\simeq \widehat{CFA}(S^3 \backslash K)\boxtimes \widehat{CFD}(T_\infty,\nu) \\
    &\xrightarrow{\mathbf{id}\boxtimes f} \widehat{CFA}(S^3\backslash K)\boxtimes \widehat{CFDA}(\mathbf{AZ}) \boxtimes \widehat{CFD}(T_\infty,\nu) \\
    &\xrightarrow{\mathbf{id}\boxtimes h_D} \widehat{CFA}(S^3 \backslash K)\boxtimes M_D
\end{split}
\]
maps the generator $1$ of $\widehat{CFK}(S^3,K)\simeq \mathbb{F}_2$ to $y\boxtimes a$. Furthermore, the chain complex 
\[
\widehat{CFA}(S^3 \backslash K)\boxtimes M_D
\]
is generated by three elements, namely $y\boxtimes a$, $y\boxtimes c$, and $y\boxtimes d$ and the differential is given by 
\[
\partial (y\boxtimes c)=m_3(y,\rho_2,\rho_1)\boxtimes e=y\boxtimes e.
\]
Hence there exists a homotopy equivalence 
\[
h_C:\widehat{CFA}(S^3\backslash K)\boxtimes M_D \rightarrow \widehat{CFK}(S^3,K)
\]
such that $(h_C\circ m)(1)=1$. However, since $\iota_{S^3 \backslash K}$ is a homotopy equivalence and any homotopy autoequivalence of $\widehat{CFK}(S^3,K)\simeq \mathbb{F}_2$ is homotopic to the identity, we should have 
\[
\tilde{f}(1)=((\mathbf{id}\boxtimes \Omega\boxtimes \mathbf{id})\circ (\iota^{-1}_{S^3 \backslash K}\boxtimes \mathbf{id})\circ (\mathbf{id}\boxtimes h^{-1}_D)\circ h^{-1}_C)((h_C\circ m)(1))=(h_C\circ m)(1)=1.
\]
Therefore $\tilde{f}$ is homotopic to the identity map. Since it is obvious that $\iota_{\mathbf{unknot}}$ is also homotopic to the identity map, we get 
\[
g_{\mathbf{unknot}} \sim F_{\mathbf{unknot}},
\]
which implies that $g$ itself should not be nullhomotopic. Since box-tensoring with $\widehat{CFDA}(\mathbf{AZ})$ is an equivalence of categories and $g$ clearly has bidegree $(0,0)$, we can apply \Cref{lem:basis} to see that $g$ should be chain homotopic to one of the following three morphisms:
\[
\mathbf{id}\boxtimes G_\infty,\,\mathbf{id}\boxtimes (G_\infty \circ \Phi^D_{\Y}\Psi^D_{\Y}),\,\mathbf{id}\boxtimes (G_\infty \circ (1+\Phi^D_{\Y}\Psi^D_{\Y})).
\]
Suppose that $g$ is homotopic to $\mathbf{id}\boxtimes (G_\infty \circ \Phi^D_{\Y}\Psi^D_{\Y})$. Then we should have 
\[
g_K \sim F_K \circ \Phi_{K\cup\mathbf{unknot},K}\Psi_{K\cup\mathbf{unknot},K} \sim \Phi_K\Psi_K \circ F_K.
\]
We have already seen that $g_{\mathbf{unknot}}$ is not nullhomotopic, which is a contradiction since $\Phi_{\mathbf{unknot}}$ and $\Psi_{\mathbf{unknot}}$ are both nullhomotopic. Therefore $g_K$ is homotopic to either $\mathbf{id}\boxtimes G_\infty$ or $\mathbf{id}\boxtimes (G_\infty \circ (1+\Phi^D_{\Y} \Psi^D_{\Y}))$, as desired.
\end{proof}

Now we are ready to prove \Cref{thm:mainthm4}.
\begin{proof}[Proof of \Cref{thm:mainthm4}]
Denote the homotopy autoequivalence of $\widehat{CFK}(S^3,K)$ defined in the theorem as $\tilde{\iota}_K$. By \Cref{lem:morphismf}, we know that $f\circ G_\infty$ is homotopic to either $(\mathbf{id}\boxtimes G_\infty)\circ \iota^{-1}_{\Y}$ or $(\mathbf{id}\boxtimes G_\infty)\circ \iota^{-1}_{\Y}\circ (1+\Phi^D_{\Y}\Psi^D_{\Y})$, so we should have either
\[
\tilde{\iota}_K \circ F_K \sim F_K \circ \iota^{-1}_{K\cup\mathbf{unknot}} \sim \iota^{-1}_K \circ F_K
\]
or
\[
\begin{split}
\tilde{\iota}_K \circ F_K &\sim F_K \circ \iota^{-1}_{K\cup\mathbf{unknot}}\circ (1+\Phi_{K\cup\mathbf{unknot},K}\Psi_{K\cup\mathbf{unknot},K}) \\
&\sim \iota^{-1}_K \circ (1+\Phi_K\Psi_K) \circ F_K \\
&\sim \iota^{-1}_K \circ \iota^2 _K \circ F_K \\
&\sim \iota_K \circ F_K.
\end{split}
\]
Since the trivial saddle cobordism from $K\cup\mathbf{unknot}$ to $K$ clearly has a right inverse, its associated cobordism map $F_K$ admits a homotopy right inverse. Hence, by precomposing with the homotopy right inverse of $F_K$, we see that $\tilde{\iota}_K$ should be homotopic to either $\iota_K$ or $\iota^{-1}_K$, as desired.
\end{proof}

\begin{rem}
\label{rem:knotinv}
The proof of the pairing theorem (\Cref{eqn:morpairing}) also works in the following way:
\[
\mathbf{Mor}(\widehat{CFD}(T_\infty,\nu),\widehat{CFD}(S^3 \backslash K))\simeq \widehat{CFK}(S^3,K).
\]
The reason is that, although $\widehat{CFD}(T_\infty,\nu)$ is not homotopy equivalent to $\widehat{CFDA}(\mathbf{AZ})\boxtimes \widehat{CFD}(T_\infty,\nu)$, $\widehat{CFK}(S^3,K)$ is homotopy equivalent to $\widehat{CFDA}(\mathbf{AZ})\boxtimes \widehat{CFK}(S^3,K)$. Hence, given an involution $\iota_M\in \mathbf{Inv}_D(S^3 \backslash K)$, one can consider the following map
\[
\begin{split}
    \widehat{CFK}(S^3,K) &\simeq \mathbf{Mor}(\widehat{CFD}(T_\infty,\nu),\widehat{CFD}(S^3 \backslash K)) \\
    &\xrightarrow{\mathbf{id}\boxtimes}\mathbf{Mor}(\widehat{CFDA}(\mathbf{AZ})\boxtimes \widehat{CFD}(T_\infty,\nu),\widehat{CFDA}(\mathbf{AZ})\boxtimes \widehat{CFD}(S^3 \backslash K)) \\
    &\xrightarrow{h\mapsto \iota_{S^3 \backslash K} \circ h \circ f} \mathbf{Mor}(\widehat{CFD}(T_\infty,\nu),\widehat{CFD}(S^3 \backslash K)) \simeq \widehat{CFK}(S^3,K)
\end{split}
\]
Here, $f$ is the type-D morphism given in \Cref{thm:mainthm4}. Following the proof of \Cref{thm:mainthm4}, it is straightforward to see that the above map is homotopic to either $\iota_K$ or $\iota^{-1}_K$. This gives a more applicable interpretation of \Cref{thm:mainthm4}, since type-D modules are easier to work with than type-A modules.
\end{rem}

\begin{exmp}
\label{exmp:trefoil}
Let $K$ be the left-handed trefoil. The knot Floer chain complex $CFK_{UV}(S^3,K)$ is generated by three elements $a,b,c$, which lie on bidegrees $(0,1),(1,2),(-1,0)$, respectively, and the differential is given as follows.
\[
\xymatrix{
b & & a\ar[ll]_{U}\ar[dd]^{V} \\
& & \\
& & c
}
\]
It is known \cite[Section 8]{hendricks2017involutive} that the action of $\iota_K$ is given by the reflection along the diagonal, i.e. fixes $a$ and exchanges $b$ and $c$.

On the bordered side, we know from \cite[Theorem 11.26]{lipshitz2018bordered} that the Floer chain complex of $K$ determines $\widehat{CFD}(S^3 \backslash K)$. Thus we see that $\widehat{CFD}(S^3 \backslash K)$ is generated by 7 elements $e_0,f_0,f_1,g_0,g_1,h_1,k_1$, where the differential is given as follows.
\[
\xymatrix{
f_0 \ar[d]_{\rho_1} & f_1 \ar[l]_{\rho_2} & e_0 \ar[l]_{\rho_3} \ar[d]^{\rho_1} \\
k_1 & & g_1 \\
& h_1 \ar[ul]_{\rho_{23}} & g_0 \ar[l]^{\rho_3} \ar[u]_{\rho_{123}}
}
\]
It can be seen via straightforward computation that there are only two homotopy classes of degree-preserving type-D endomorphisms of $\widehat{CFD}(S^3 \backslash K)$, represented by 0 and $\mathbf{id}$. Hence $\widehat{CFD}(S^3 \backslash K)$ is \emph{homotopy-rigid}, i.e. it admits a unique homotopy class of homotopy autoequivalences. This means that there exists only one homotopy class of homotopy equivalences
\[
\widehat{CFDA}(\mathbf{AZ})\boxtimes \widehat{CFD}(S^3 \backslash K)\rightarrow \widehat{CFD}(S^3 \backslash K).
\]
Since one of such homotopy equivalences can be computed explicitly using the proof of \cite[Theorem 37]{hanselman2018heegaard}, we deduce that it also gives an explicit description of $\iota_{S^3 \backslash K}$. Applying \Cref{thm:mainthm4} then recovers the hat-flavored action
\[
\iota_K(a)=a,\,\iota_K(b)=c,\,\iota_K(c)=b
\]
in $\widehat{CFK}(S^3,K)$, which is consistent with the action of $\iota_K$ on $CFK_{UV}(S^3,K)$.
\end{exmp}

\begin{rem}
In general, one can prove that $\widehat{CFD}(S^3 \backslash K)$ is homotopy-rigid whenever $K$ is an L-space knot, which means that one can explicitly compute $\iota_{S^3 \backslash K}$ for such knots by computing the box tensor product $\widehat{CFDA}(\mathbf{AZ})\boxtimes \widehat{CFD}(S^3 \backslash K)$ and finding a sequence of homotopy equivalences which connects it to $\widehat{CFD}(S^3 \backslash K)$. One can check using \Cref{thm:mainthm4} that the hat-flavored action of $\iota_K$ is given by ``reflection with respect to the diagonal''. This is consistent with the action of $\iota_K$ on $CFK_{UV}(S^3,K)$, which was first determined in \cite[Section 7]{hendricks2017involutive}.
\end{rem}

\Cref{thm:mainthm4} can also be used in the reverse way to compute $\iota_{S^3 \backslash K}$ from $\iota_K$, as shown in \Cref{exmp:figeight}.
\begin{exmp}
\label{exmp:figeight}
Let $K$ be the figure-eight knot. The knot Floer chain complex $CFK_{UV}(S^3,K)$ is generated by five elements $a,b,c,d,x$, which lie on bidegrees $(0,0),(1,1),(-1,-1),(0,0),(0,0)$, respectively, and the differential is given as follows.
\[
\xymatrix{
b \ar[dd]_{V} & & a\ar[ll]_{U}\ar[dd]^{V} & & \\
& & & \oplus & x \\
d & & c\ar[ll]^{U} &&
}
\]
Furthermore, the involution $\iota_K$ is given by 
\[
\iota_K(a)=a+x,\,\iota_K(b)=c,\,\iota_K(c)=b,\,\iota_K(d)=d,\,\iota_K(x)=x+d.
\]
On the other hand, $\widehat{CFD}(S^3 \backslash K)$ is generated by 9 elements $e_0,f_0,g_0,h_0,e_1,f_1,g_1,h_1,z$, where the differential is given as follows.
\[
\xymatrix{
f_0 \ar[d]_{\rho_1} & e_1 \ar[l]_{\rho_2} & e_0 \ar[l]_{\rho_3} \ar[d]^{\rho_1} & & \\
f_1 & & h_1 & \oplus & z \ar@(ur,dr)^{\rho_{12}} \\
g_0 \ar[u]^{\rho_{123}} & g_1\ar[l]^{\rho_2} & h_0\ar[l]^{\rho_3} \ar[u]_{\rho_{123}} & &
}
\]

Unlike the trefoil case (covered in \Cref{exmp:trefoil}), the type-D module $\widehat{CFD}(S^3 \backslash K)$ is not homotopy-rigid, so we cannot find a random homotopy equivalence between $\widehat{CFDA}(\mathbf{AZ})\boxtimes \widehat{CFD}(S^3 \backslash K)$ and $\widehat{CFD}(S^3 \backslash K)$ and claim that it is homotopic to $\iota_{S^3 \backslash K}$. Denote by $M$ and $N$ the type-D submodule of $\widehat{CFD}(S^3 \backslash K)$ generated by $z$ and everything else (i.e. $e_0,\cdots,h_1$), respectively, so that we have a splitting 
\[
\widehat{CFD}(S^3 \backslash K) \simeq M\oplus N.
\]
Using the proof of \cite[Theorem 37]{hanselman2018heegaard}, one can explicitly construct homotopy equivalences 
\[
\begin{split}
    F_M &: \widehat{CFDA}(\mathbf{AZ})\boxtimes M \rightarrow M, \\
    F_N &: \widehat{CFDA}(\mathbf{AZ})\boxtimes N \rightarrow N.
\end{split}
\]
Consider $F=F_M \oplus F_N:\widehat{CFDA}(\mathbf{AZ})\boxtimes \widehat{CFD}(S^3 \backslash K)\rightarrow \widehat{CFD}(S^3 \backslash K)$. Then $F\circ \iota^{-1}_{S^3 \backslash K}$ is a homotopy autoequivalence of $\widehat{CFD}(S^3 \backslash K)$. Recall that we have a pairing theorem 
\[
\mathbf{Mor}(\widehat{CFD}(S^3\backslash K),\widehat{CFD}(S^3 \backslash K))\simeq \widehat{CF}(-(S^3 \backslash K) \cup (S^3 \backslash K)) \simeq \widehat{CF}(S^3 _{0}(K\sharp -K)).
\]
Since $F$ is a homotopy equivalence, it should correspond to a nontrivial element with absolute $\mathbb{Q}$-grading $\frac{1}{2}$ in $\widehat{HF}(S^3 _0 (K \sharp -K),\mathfrak{s}_0)$, where $\mathfrak{s}_0$ denotes the unique spin structure on $S^3 _0 (K \sharp -K)$. The integral surgery formula for knots \cite[Theorem 1.1]{ozsvath2008knot} tells us that the $\frac{1}{2}$-graded piece $V_{\frac{1}{2}}$ of $\widehat{HF}(S^3 _0 (K \sharp -K),\mathfrak{s}_0)$ is 5-dimensional.

Now we construct an explicit basis of $V_{\frac{1}{2}}$ in terms of type-D endomorphisms of $\widehat{CFD}(S^3 \backslash K)$. Consider the type-D endomorphisms $K_1,K_2,K_3$ of $\widehat{CFD}(S^3 \backslash K)$, defined as 
\[
\begin{split}
    K_1(e_0)=z,\,K_1(h_1)=\rho_2 z,\, & K_1(\text{everything else})=0, \\
    K_2(z)=g_0+\rho_3 f_1,\, & K_2(\text{everything else})=0, \\
    K_3(z)=z,\, & K_3(\text{everything else})=0.
\end{split}
\]
We claim that the type-D morphisms $\mathbf{id}$, $K_1$, $K_2$, $K_2 \circ K_1$, and $K_3$ are linearly independent up to homotopy and thus form a basis of $V_{\frac{1}{2}}$. To prove the claim, we take a tensor product with $\widehat{CFA}(T_\infty,\nu)$, and consider the maps $\mathbf{id}\boxtimes K_1$ and $\mathbf{id}\boxtimes K_2$, which are now considered as chain endomorphisms of $\widehat{CFK}(S^3,K)$. One can easily see that 
\[
\begin{split}
(\mathbf{id}\boxtimes K_1)(a)=x,\, & (\mathbf{id}\boxtimes K_1)(\text{everything else})=0, \\
(\mathbf{id}\boxtimes K_2)(x)=d,\, & (\mathbf{id}\boxtimes K_2)(\text{everything else})=0, \\
(\mathbf{id}\boxtimes K_3)(x)=x,\, & (\mathbf{id}\boxtimes K_3)(\text{everything else})=0.
\end{split}
\]
Hence we see that $\mathbf{id}\boxtimes g$ for $g=\mathbf{id},K_1,K_2,K_2\circ K_1,K_3$ induce linearly independent endomorphisms of $\widehat{HFK}(S^3,K)$, and so the claim is proven.

Given a type-D morphism $m:\widehat{CFDA}(\mathbf{AZ})\boxtimes \widehat{CFD}(S^3\backslash K)\rightarrow \widehat{CFD}(S^3 \backslash K)$, we define an endomorphism $E_m$ of $\widehat{CFK}(S^3,K)$ as follows.
\[
\begin{split}
    E_m : \widehat{CFK}(S^3 ,K) &\simeq \mathbf{Mor}(\widehat{CFD}(T_\infty,\nu),\widehat{CFD}(S^3 \backslash K)) \\
    &\xrightarrow{\mathbf{id}\boxtimes} \mathbf{Mor}(\widehat{CFDA}(\mathbf{AZ})\boxtimes \widehat{CFD}(T_\infty,\nu),\widehat{CFDA}(\mathbf{AZ})\boxtimes \widehat{CFD}(S^3 \backslash K)) \\
    &\xrightarrow{h\mapsto m \circ  h \circ f} \mathbf{Mor}(\widehat{CFD}(T_\infty,\nu),\widehat{CFD}(S^3 \backslash K)) \simeq \widehat{CFK}(S^3,K)
\end{split}
\]
Here, $f$ denotes the type-D morphism appearing in \Cref{thm:mainthm4}. Then a manual computation tells us that, for the homotopy equivalence $F$ described above, the endomorphism $E_F$ acts on $\widehat{CFK}(S^3,K)$ by
\[
a\mapsto a,\,b\mapsto c,\,c\mapsto b,\,d\mapsto d,\,x\mapsto x.
\]
Comparing this with $\iota_K$, we see that $\mathbf{id}\boxtimes (F\circ \iota^{-1}_{S^3 \backslash K})$ acts on $\widehat{CFK}(S^3,K)$ by 
\[
a\mapsto a+x,\,x\mapsto x+d,\,b\mapsto b,\,c\mapsto c,\,d\mapsto d.
\]
Since $F\circ \iota^{-1}_{S^3 \backslash K}$ is an element of $V_{\frac{1}{2}}$, which is generated by $\mathbf{id}$, $K_1$, $K_2$, $K_2 \circ K_1$, and $K_3$, we deduce that 
\[
\iota_{S^3 \backslash K} \sim (\mathbf{id} +K_1 + K_2)\circ F.
\]
\end{exmp}

\bibliographystyle{amsalpha}
\bibliography{main}
\end{document}